\documentclass[12pt]{amsart}

\usepackage[utf8]{inputenc}
\usepackage[T1]{fontenc}
\usepackage{amsfonts}
\usepackage{amsmath}
\usepackage[mathscr]{eucal}
\usepackage{amssymb}
\usepackage{amstext}
\usepackage{amsthm}
\usepackage{mathrsfs}
\usepackage{centernot}
\usepackage{enumitem}
\usepackage{graphicx, color}
\usepackage{caption}
\usepackage{subcaption}
\usepackage{url}
\usepackage{float}
\usepackage[a4paper,margin=2cm]{geometry}

\AtBeginDocument{          }

\numberwithin{equation}{section}

\newtheorem{theorem}{Theorem}[section]
\newtheorem{corollary}[theorem]{Corollary}
\newtheorem{lemma}[theorem]{Lemma}
\newtheorem{proposition}[theorem]{Proposition}

\makeatletter
\newtheoremstyle{boldremark}
  {\topsep}   
  {\topsep}   
  {}          
  {}          
  {\bfseries} 
  {.}         
  {.5em}      
  {}          
\makeatother

\theoremstyle{boldremark}
\newtheorem{remark}[theorem]{Remark}

\title{Discrete space-time wave kernels on regular trees}

\author{Amar Ba\v{s}i\'{c}}
\address{Faculty of Electrical Engineering, University of Sarajevo, Zmaja od Bosne bb, Sarajevo, 71000, Bosnia and Herzegovina}
\email{abasic@etf.unsa.ba}

\author{Lejla Smajlovi\'{c}}
\address{Department of Mathematics and Computer Science, University of Sarajevo, Zmaja od Bosne 35, Sarajevo, 71000, Bosnia and Herzegovina}
\email{lejlas@pmf.unsa.ba}

\author{Zenan \v{S}abanac}
\address{Department of Mathematics and Computer Science, University of Sarajevo, Zmaja od Bosne 35, Sarajevo, 71000, Bosnia and Herzegovina}
\email{zsabanac@pmf.unsa.ba}

\date{}



\begin{document}

\begin{abstract}
We study the forward discrete space-time wave equation on the homogeneous $(q+1)$-regular tree $T_{q+1}$ associated with a two-parameter generalized Laplacian. Under the natural nonnegativity assumption on this operator, we derive explicit formulas for the two fundamental wave kernels. The formulas are given in terms of discrete $I$-Bessel functions and yield convolution representations for solutions with general initial conditions.

In the boundary case corresponding to the bottom of the spectrum, we obtain another explicit representation of the wave kernel in terms of discrete $J$-Bessel functions. This representation leads to a discrete analogue of the classical $I\!\leftrightarrow\!J$ relation. We also perform both analytic and numerical studies of the asymptotic behavior of the wave kernels, including large radial distance, large time, and large degree of the tree.

An important feature of our analysis is that the wave kernels are expressed as finite sums; hence, the propagation formulas remain finite for every discrete time.
\end{abstract}

\maketitle

\section{Introduction and statement of the main result}\label{sec1}

\subsection{Discrete space-time wave equation}

The classical wave equation, which governs wave propagation at speed $c>0$ in space-time $(x,t)$, is one of the most important partial differential equations and has been thoroughly studied in various space-time settings, as well as with different initial and boundary conditions. In general terms, a wave-type equation is the equation $\frac{\partial^2 }{\partial t^2} w(x;t)=c^2 \Delta_x w(x;t)$, where $\Delta_x$ is the Laplace-Beltrami operator acting on the space variable $x$.

Many studies have been devoted to the so-called semidiscrete settings, meaning settings in which the space variable is discrete and the time variable is continuous. The asymptotic behavior of solutions of the semilinear wave equation on networks was studied in \cite{Ch22}; Chitour, Mazanti, and Sigalotti \cite{CMS16} addressed the question of robustness of wave propagation on networks, while the stabilization of the wave equation on general 1D networks was studied in \cite{VZ09}. Anker, Martinot, Pedon, and Setti \cite{AMPS13} examined shifted wave equations on spaces with similar geometric structures, highlighting analogies with symmetric spaces.

Semidiscrete wave equations belong to a broad class of semidiscrete evolution equations which have recently been extensively studied. For example, these equations with the spatial variable belonging to $\mathbb Z$ have been studied in \cite{G-CLM21, LM-A23, LM-A24}. An interesting operator-theoretic method was developed in \cite{G-CLM21}. Namely, by viewing the underlying operator as a convolution with an element from the Banach algebra $\ell^1(\mathbb Z)$, using properties of Banach algebras and cosine families the authors provided explicit representations of fundamental solutions to those equations and analyzed their qualitative behavior. The convolution-type operators were used in \cite{LM-A23} and \cite{LM-A24} to derive explicit fundamental solutions to a wide class of linear semidiscrete equations including the wave equation (with a spatial variable belonging to $\mathbb Z$). Fractional damped heat and wave equations in continuous time, with the spatial variable belonging to an infinite lattice, were studied in \cite{G-C25}, where a representation formula for the solution was derived.

In this paper, we study the generalized wave equation on $(q+1)$-regular (homogeneous) trees.
The study of wave equations on homogeneous trees has been addressed from various perspectives in the literature. Pagliacci \cite{P93} provided one of the earliest treatments, outlining the basic analytical framework and solutions in the continuous-time setting. Medolla and Setti \cite{MS99} developed a spectral theory framework for the wave equation, emphasizing the role of the Laplacian’s spectral decomposition in constructing and analyzing solutions; see also \cite{Me99} for a follow-up study of the asymptotic equipartition of the energy. Some recent advances related to the wave equation on trees can be found in \cite{TS16} and \cite{CNT23}.

Discretization of the background space-time is a standard procedure in physics. The interested reader is referred to \cite{tH16} and references therein for motivation stemming from quantum-gravitational considerations, \cite{KS22} for the dimensional deconstruction model involving discrete-time heat kernels, to \cite{AG-CR23, CKSS25, DG-CR26} for recent analytical studies of discrete-time heat kernels on discrete structures and to \cite{AAD22} for large time behavior of the discretized heat kernel when the spatial space is $\mathbb R^N$.

\medskip

Time discretization of the wave equation on homogeneous trees was carried out in \cite{CP94} by replacing the second derivative in time with the second-order averaging operator $\Delta_2 f(t)= \frac{1}{2}(f(t+1)-2f(t)+f(t-1))$. Slav\'ik \cite{Slav17, Slav18} studied the wave equation on some discrete spaces in general time scales, while the authors in \cite{BSS24} introduced discrete analogues of Bessel functions and used them to derive explicit representations of solutions to the discrete wave equation on the homogeneous $2$-tree $T_2=\mathbb{Z}$. In \cite{BSS24}, it was further discussed how different discretisations of the second time derivative (forward, backward, or mixed) lead to different time asymptotics of solutions.

\subsection{Our setup} The $(q+1)$-regular homogeneous tree (also known as the Bethe lattice with coordination number $q+1$) is a connected $(q+1)$-regular graph with no cycles. We denote such a tree by $T_{q+1}$.

The combinatorial Laplacian $\Delta_{T_{q+1}}$, acting on functions $f:V(T_{q+1})\to\mathbb{R}$, where $V(T_{q+1})$ denotes the vertex set of the tree, is given by
\begin{equation*}
    \Delta_{T_{q+1}} f(x) = (q+1)\, f(x) - \sum_{x \sim y} f(y),
\end{equation*}
where the second sum is taken over all vertices $y$ connected to $x$ by an edge. 

In this paper, we study the wave equation associated with the generalized Laplacian
$$\Delta_{q+1}^{a,b}:= b\Delta_{T_{q+1}} - aI,$$
where the real parameters $a,\,b$ satisfy $b\neq0$ and 
\begin{equation}\label{eq. main ass a,b}
b(q+1)-a\geq 2|b|\sqrt{q}.
\end{equation}
In view of the fact that the $L^2$-spectrum of the Laplacian on the tree $T_{q+1}$ is purely continuous and equal to $
[q+1-2\sqrt{q},q+1+2\sqrt{q}],
$
(see \cite{MW89}), the assumptions imposed on the parameters $a,\,b$ ensure that $\Delta_{q+1}^{a,b}$ is a semipositive self-adjoint operator. Here and throughout, by semipositive we mean nonnegative self-adjoint, or, equivalently, positive semidefinite operator. Indeed, the $L^2$-spectrum of $\Delta_{q+1}^{a,b}$ is given by
\[
\sigma_2\!\left(\Delta^{a,b}_{q+1}\right)
=
b\,\sigma_2\!\left(\Delta_{q+1}\right)-a
=
\bigl[
b(q+1)-a-2|b|\sqrt q,\;
b(q+1)-a+2|b|\sqrt q
\bigr].
\]
Therefore, condition \eqref{eq. main ass a,b} 
is precisely the condition that the $L^2$-spectrum of $\Delta_{q+1}^{a,b}$ is nonnegative. This condition therefore identifies the natural boundary case $a=b(q+1) -2|b|\sqrt{q}$, in which the lower endpoint of the spectrum is zero.

Parameters $a,\, b$ have a simple interpretation: the coefficient $b$ fixes the spatial coupling strength, or diffusion/wave-speed scale, while the shift $a$ introduces a constant potential or spectral shift. Thus the family $\Delta_{q+1}^{a,b}=b\Delta_{T_{q+1}}-aI$ provides a unified way to treat several natural graph Laplacians and shifted Laplace-type operators within the same discrete wave equation.

Namely, for a $(q+1)$-regular graph $X$, the family $\Delta_{X}^{a,b}=b((q+1)I-A_X)-aI$, where $A_X$ is the adjacency operator and parameters $a,\,b$ satisfy \eqref{eq. main ass a,b} contains, as special cases, the following  Laplace-type operators: the normalized (probabilistic) Laplacian $\Delta_X^{\rm norm} = \frac{1}{q+1}\Delta_X$; the signless Laplacian $\Delta_X^{+}=(q+1)I+A_X= -\Delta_X +2(q+1)I$  studied in \cite{CRS07} and the modified Laplacian $\mathcal L_X=\frac{1}{q+1}\left( \Delta_X - (\sqrt{q}-1)^2I\right)$ studied in \cite{MS99}.

Generalizations of the Laplacian with $b=1$ and special choices $a<0$ can also be found in the study of electrical networks \cite{DS-PB18}. The operator $\Delta_{q+1}^{q+1-2\sqrt{q},1}$ is relevant in physics and called the gapless Laplacian (see, e.g., \cite{BDST26} where it is shown that the scalar field theories associated with the gapless Laplacian on the Bethe lattice exhibit novel critical behavior).

Moreover, the operator $\Delta_{q+1}^{-(\sqrt{q}-1)^2/(2\sqrt{q}), 1/(2\sqrt{q})} $ was studied in equation (16) of \cite{AMPS13}, which in turn can be rewritten in terms of the operator $\mathcal{T}_{q}=\frac{1}{\sqrt{q}}A_{q+1}$ studied in \cite{AS19, BL13}. Constant-potential Schrödinger operators on graphs $X$ of the form $\Delta_X+\lambda I$, which arise naturally in spectral theory and mathematical physics, are recovered as special cases of $\Delta_{q+1}^{a,b}$ by choosing $b=1$ and $a=-\lambda$, where $\lambda \geq -(\sqrt{q}-1)^2$; see, for example, \cite{KS14}. 

\subsection{The generalized wave equation} 
Let $\mathbb{T}$ be a time scale, and let $x_0 \in V(T_{q+1})$ be an arbitrary but fixed vertex, referred to as the \textit{root} of the tree $T_{q+1}$. 
Let $a, b$ be real parameters such that $b\neq 0$ and \eqref{eq. main ass a,b} holds.
The \textit{generalized wave equation}
\begin{equation}\label{wave_eq_disc}
\Delta_{q+1}^{a,b}\, W(x_0, x; t)
\;+\;
\partial_{t,\mathbb{T}}^2 W(x_0, x; t)
= 0
\end{equation}
possesses two linearly independent \textit{fundamental solutions}
$$
W_{q+1,\mathbb{T}}^{a,b},\, V_{q+1,\mathbb{T}}^{a,b}: V(T_{q+1}) \times V(T_{q+1}) \times \mathbb{T} \to \mathbb{R}
$$
subject to two natural initial conditions determined by the time scale $\mathbb{T}$. The functions $W_{q+1,\mathbb{T}}^{a,b}$ and $V_{q+1,\mathbb{T}}^{a,b}$ are called the \textit{wave kernels} on the tree $T_{q+1}$ in the time scale $\mathbb{T}$.\footnote{For a definition of the operator $\partial_{t,\mathbb{T}}$ on a general time scale $\mathbb{T}$, we refer the interested reader to Bohner and Peterson \cite{BP1}, where it is denoted by $f^{\Delta}(t)$.}

In particular, when $\mathbb{T} = \mathbb{R}$, the delta derivative $\partial_{t,\mathbb{T}}$ reduces to the usual derivative, so that $\partial_{t,\mathbb{R}} f(t) = f'(t)$. When $\mathbb{T} = \mathbb{N}_0$, it coincides with the forward difference operator
\[
\partial_{t,\mathbb{N}_0} f(t) = f(t+1) - f(t),
\]
and its second iterate is the second forward difference
\[
\partial_{t,\mathbb{N}_0}^2 f(t) = f(t+2) - 2f(t+1) + f(t).
\]
For brevity, in the case $\mathbb{T} = \mathbb{N}_0$ we will simply write $\partial_t$ instead of $\partial_{t,\mathbb{N}_0}$, since this is the only time derivative we will consider. We will also omit the notation for the time scale $\mathbb{T}$, as throughout this paper we fix $\mathbb{T} = \mathbb{N}_0$.

For $\mathbb{T}=\mathbb{N}_0$, the wave kernels $W_{q+1}^{a,b}$ and $V_{q+1}^{a,b}$ are characterized by the following initial conditions: 
\begin{equation}\label{cond1}
W_{q+1}^{a,b}(x_0, x; 0) =\delta_{x_0}(x):=
\begin{cases}
1, & \text{if } x = x_0, \\
0, & \text{if } x \neq x_0,
\end{cases} \quad \text{and}\quad \partial_t W_{q+1}^{a,b}(x_0, x; 0) = 0,\,\, \forall x\in V(T_{q+1}),
\end{equation}
and
\begin{equation}\label{cond on V}
V_{q+1}^{a,b}(x_0, x; 0) = 0 \quad \text{and}\quad  \partial_t  V_{q+1}^{a,b}(x_0, x; 0) =
\begin{cases}
1, & \text{if } x = x_0, \\
0, & \text{if } x \neq x_0,
\end{cases} \,\, \forall x\in V(T_{q+1}).
\end{equation}

The solutions of \eqref{wave_eq_disc} subject to \eqref{cond1} and \eqref{cond on V}, are the \textit{discrete wave kernels} associated with the operator $\Delta_{q+1}^{a,b}$ on the tree $T_{q+1}$.

By invariance under graph automorphisms and the distance-transitivity of the
$(q+1)$-regular tree, the wave kernels depend only on the distance $r=|x|=d(x_0,x)$ from the root $x_0$; hence, they can be written as $W_{q+1}^{a,b}(r;t)$ and $V_{q+1}^{a,b}(r;t)$, respectively.

\subsection{Main results}

The main result of this paper is Theorem \ref{Tqkernel_a}, in which an explicit expression for the discrete-time wave kernels on $T_{q+1}$ (with a fixed base point $x_0$), associated with the generalized Laplacian $\Delta_{q+1}^{a,b}$, is given in terms of discrete $I$-Bessel functions $I_{n}^{c}(k)$ introduced in Section \ref{sec:int-repr} below\footnote{Discrete Bessel-type functions also appear naturally in explicit solution formulas for related semidiscrete damped and evolution equations; see, for example, \cite{LM-A23,LM-A24}.}.
\begin{theorem}\label{Tqkernel_a}
Let $q\geq 1$ be fixed, let $a,\, b$ be real numbers with $b\neq 0$ satisfying \eqref{eq. main ass a,b}. Set
$$c_{a,b} = \dfrac{2b\sqrt{q}}{b(q+1) - a}, \quad d_{a,b}=b(q+1)-a.$$ 
The solution $W^{a,b}_{q+1}(r;t)$ to the initial value problem \eqref{wave_eq_disc}
satisfying initial conditions \eqref{cond1}
is
\begin{equation}\label{wave_kernel_a}
W^{a,b}_{q+1}(r;t)
=\sum_{k=0}^{\lfloor \frac{t}{2} \rfloor} (-1)^{k+r}\, \binom{t}{2k}\, d_{a,b}^k
\left(
q^{-\tfrac{r}{2}}\, I_r^{c_{a,b}}(k)
- (q - 1) \sum_{\ell=1}^{\lfloor \tfrac{k - r}{2} \rfloor} q^{-\tfrac{r+ 2\ell}{2}}\, I_{r + 2\ell}^{c_{a,b}}(k)
\right).
\end{equation}
The solution $V^{a,b}_{q+1}(r;t)$ of the initial value problem \eqref{wave_eq_disc}
satisfying initial conditions \eqref{cond on V} is
\begin{equation}\label{wave_kernel_b}
V^{a,b}_{q+1}(r;t)
=\sum_{k=0}^{\lfloor \frac{t-1}{2} \rfloor} (-1)^{k+r}\, \binom{t}{2k+1}\, d_{a,b}^k
\left(
q^{-\tfrac{r}{2}}\, I_r^{c_{a,b}}(k)
- (q - 1) \sum_{\ell=1}^{\lfloor \tfrac{k - r}{2} \rfloor} q^{-\tfrac{r+ 2\ell}{2}}\, I_{r + 2\ell}^{c_{a,b}}(k)
\right).
\end{equation}
\end{theorem}

For each fixed time $t$, the sums in \eqref{wave_kernel_a} and \eqref{wave_kernel_b} are finite. Moreover, since $I_n^c(k)=0$ whenever $n>k$, the kernels have finite propagation in the radial variable: for fixed $t$, only distances $r\leq \lfloor t/2\rfloor$ can contribute to $W_{q+1}^{a,b}(r;t)$, while only distances $r\leq \lfloor (t-1)/2\rfloor$ can contribute to $V_{q+1}^{a,b}(r;t)$. Thus the formulas give effective finite expressions for the wave kernels at every discrete time. In the special case $q=1$, the tree $T_2$ is identified with the integer lattice $\mathbb Z$, the factor $q-1$ makes the second sum vanish, and the formulas specialize to the corresponding one-dimensional discrete-time wave setting previously studied in \cite{BSS24, Slav17,Slav18}.

We study the boundary case $a=b(q+1-2\sqrt{q})$, $b>0$ in greater detail. First, in  Proposition \ref{prop:boundary} below, we derive another expression for the wave kernels in terms of finite sums of the discrete $J$-Bessel functions. Using uniqueness of the wave kernel under given initial conditions, in Corollary  \ref{cor: I and J rel} we prove the identity 
  \begin{equation}\label{discrete_analog}
        J_{2n}^{2c}(t)=\sum_{k=0}^{\lfloor t/2 \rfloor}(-1)^{k+n} (2c^2)^k \binom{t}{2k} I_{n}^{1}(k)
\end{equation}
relating discrete $I$-Bessel and $J$-Bessel functions (defined in \eqref{eq:def-I-Bessel} and \eqref{eq:def-J-Bessel} below), which can be viewed as a discrete analogue of the classical $I\leftrightarrow J$ relation.

In Section \ref{sec. asymptotics} below we study asymptotic properties of the boundary wave kernel $W_{q+1}^{b(q+1-2\sqrt{q}),b}$ and prove that:
\begin{itemize}
    \item For a fixed time $t\in\mathbb{N}_0$ and $q\geq 1$,
$$
\lim_{r\to\infty}W_{q+1}^{b(q+1-2\sqrt{q}),b}(r;t)=0;
$$
\item For a fixed radial distance $r\in\mathbb{N}_0$ and $q\geq 1$, as $t\to\infty$
$$
W_{q+1}^{b(q+1-2\sqrt{q}),b}(r;t)\sim \frac{(-1)^rq^{-r/2}}{\sqrt{\pi t\sin\theta}} C_b(r;t)\left( 1+4b\sqrt{q}
\right) ^{\frac{t}{2}}\cos \left( \left( t+\frac{1}{2}\right)
\theta -\frac{ \pi}{4} \right),
$$
where $\theta\in(0,\pi)$ is defined by $\cos\theta=(1+4b\sqrt{q})^{-1/2}$ and $C_b(r;t)$ is a certain function which is bounded from below as $t\to\infty$;
\item For fixed $r,\, t \in \mathbb{N}_0$ with $t \ge 2r$, as $q \to \infty$,
$$
W^{1-\frac{2\sqrt{q}}{q+1}, \frac{1}{q+1}}_{q+1}(r;t)
= \binom{t}{2r}(q+1)^{-r}\left(1+O\left(\frac{1}{\sqrt{q}}\right)\right),
$$
and in particular, for $r>0$
$$
\lim_{q \to \infty} W^{1-\frac{2\sqrt{q}}{q+1}, \frac{1}{q+1}}_{q+1}(r;t) = 0.
$$
\end{itemize}
We conduct numerical analysis of the behavior of the wave kernel $W_{q+1}^{a,b}$, supported by numerical tables and graphical illustrations, in order to complement the theoretical results and to demonstrate the predicted asymptotic behavior in concrete examples of generalized Laplacian. 

In Proposition \ref{prop. arb param} below, we prove that, for any choice of real numbers $a,\, b$ with $b\neq 0$ solutions to \eqref{wave_eq_disc} with initial conditions \eqref{cond1} and \eqref{cond on V} can be computed explicitly. Therefore, it is not necessary to assume semipositivity of $\Delta_{q+1}^{a,b}$ in order to \emph{solve explicitly} the wave equation. However, in Remark \ref{rem. why semipos} we provide evidence that, when semipositivity is violated, the fundamental solution to \eqref{wave_eq_disc} satisfying the initial condition \eqref{cond1} does not necessarily behave like a \emph{wave}; instead it may exhibit exponential growth, hence the equation \eqref{wave_eq_disc} can no longer be interpreted as a wave equation in the usual sense. Moreover, semipositivity of $\Delta_{q+1}^{a,b}$ implies uniqueness of fundamental solutions of \eqref{wave_eq_disc} with initial conditions \eqref{cond1} and \eqref{cond on V} and is used to deduce the identity \eqref{discrete_analog}.

\subsection{Organization of the paper}
The paper is organized as follows. In Section \ref{sec. wave on tree}, we describe wave kernels on the $(q+1)$-regular tree in terms of the convolution kernel. Section \ref{sec:int-repr} is devoted to the study of discrete $I$-Bessel and $J$-Bessel functions; in particular, we deduce integral representations of those functions which will be used in Section \ref{sec. wave on tree sol} to prove the main theorem and describe explicitly solutions to the wave equation on $T_{q+1}$ with general initial conditions. In Section \ref{sec. applic on tree}, we study the boundary case $a=b(q+1-2\sqrt{q})$, derive a different, explicit expression for the wave kernels in terms of the discrete $J$-Bessel functions and prove \eqref{discrete_analog}. Section \ref{sec. asymptotics} is devoted to the study of the asymptotic behavior of the wave kernels. In the last section, we derive fundamental solutions to the second-order evolution equation \eqref{wave_eq_disc} associated with the generalized Laplacian, which is not semipositive. The numerical tables supporting the graphical illustrations are collected in the Appendix.

\section{Generalized wave operators} \label{sec. wave on tree}

In this section, we study the generalized wave equation \eqref{wave_eq_disc}. First, we identify two wave operators on the tree. Exploiting its symmetry and recursive structure, the analysis can be reduced to radial functions around the base point $x_0$; hence, the wave operators admit a natural description as convolution operators, as described below.

\subsection{Wave operators}

The definition of the discrete-time wave operators is motivated by the series
representation of the discrete cosine function (see \cite[Theorem~3.5.2]{Cu15})
\begin{equation}\label{eq. cos alpha}
\cos_\alpha(t,0)
=
\sum_{k=0}^\infty (-1)^k h_{2k}(t,0)\alpha^{2k}
=
\sum_{k=0}^\infty (-1)^k \alpha^{2k}\binom{t}{2k},
\quad \alpha\in\mathbb C,
\end{equation}
on the time scale $\mathbb{T}=\mathbb{N}_0$, where $h_{k}(t,0) = \binom{t}{k}$, with the convention that $\binom{t}{k}=0$ when $k>t$. The functions $h_k(t,0)$ are the discrete-time monomials (also called time-scale monomials) on the time scale $\mathbb{T}=\mathbb{N}_0$, i.e., the canonical polynomial basis associated with the forward difference operator.

We recall that the discrete cosine function is usually introduced in the time-scale setting through the time-scale exponential; see, for example, \cite[Definition~3.25]{BP1}. While this standard approach involves regressivity assumptions ($\alpha\neq \pm i$), here we accept the non-regressive subscript and use the equation \eqref{eq. cos alpha} as a polynomial extension of the standard discrete cosine. This expression is well-defined for every $\alpha\in\mathbb C$, because the sum is finite for each fixed $t\in\mathbb N_0$.

This leads to the definition of the discrete-time wave operators associated with the generalized Laplacian $\Delta_{q+1}^{a,b}$:
\begin{equation*}
\mathcal{W}_t^{a,b} \;=\; \sum_{k=0}^{\infty} (-1)^k\, h_{2k}(t,0)\, (\Delta_{q+1}^{a,b})^{\,k}, \quad \text{and} \quad \mathcal{V}_t^{a,b} \;=\; \sum_{k=0}^{\infty} (-1)^k\, h_{2k+1}(t,0)\, (\Delta_{q+1}^{a,b})^{\,k},
\end{equation*}
where $(\Delta_{q+1}^{a,b})^{\,k}$ denotes the $k$-th power of the operator $\Delta_{q+1}^{a,b}$. The operators $\mathcal{W}_t^{a,b}$  and $\mathcal{V}_t^{a,b}$ are well-defined, since the sums are finite, and bounded on $L^2(T_{q+1})$ with respect to the operator norm
$\|\cdot\|_{2}$ corresponding to the counting measure.

In particular, the norm $\|\mathcal{W}_t^{a,b}\|_{2}$  on $L^{2}(T_{q+1})$ is bounded by $\left(1+\sqrt{|b(q+1)-a|+2|b|\sqrt{q}}\right)^{t}$ for $t\geq 1$. Namely, 
$$
\big\|\mathcal W_t^{a,b}\big\|_{2}
\;\le\;
\sum_{k=0}^{\lfloor t/2\rfloor}\binom{t}{2k}\,\|\Delta_{q+1}^{a,b}\|_{2}^{\,k}= \sum_{k=0}^{\lfloor t/2\rfloor}\binom{t}{2k}\left(|b(q+1)-a|+2|b|\sqrt{q}\right)^k,
$$
since, as observed above,
$$
\|\Delta_{q+1}^{a,b}\|_{2}=\sup_{\lambda\in\sigma_2\left(\Delta_{q+1}^{a,b}\right)}
|\lambda|= |b(q+1)-a|+2|b|\sqrt{q}.
$$
By separating the even powers in the binomial expansions of $(1+z)^t$ and
$(1-z)^t$, we obtain
\[
\sum_{k=0}^{\lfloor t/2\rfloor}\binom{t}{2k}z^{2k}
=
\frac{(1+z)^t+(1-z)^t}{2}.
\]
Therefore, with $z=\sqrt{\|\Delta_{q+1}^{a,b}\|_{2}}$,
\[
\big\|\mathcal W_t^{a,b}\big\|_2
\leq
\frac{
\bigl(1+\sqrt{\|\Delta_{q+1}^{a,b}\|_{2}}\bigr)^t
+
\bigl(1-\sqrt{\|\Delta_{q+1}^{a,b}\|_{2}}\bigr)^t
}{2}
\leq
\left(1+\sqrt{|b(q+1)-a|+2|b|\sqrt{q}}\right)^t .
\]
Similarly,
\[
\big\|\mathcal V_t^{a,b}\big\|_2
\leq
\sum_{k=0}^{\lfloor (t-1)/2\rfloor}
\binom{t}{2k+1}\|\Delta_{q+1}^{a,b}\|_{2}^{\,k}
=
\sum_{k=0}^{\lfloor (t-1)/2\rfloor}
\binom{t}{2k+1}z^{2k}.
\]
Using the odd-power part of the binomial expansion, we obtain, for $z>0$,
\[
\sum_{k=0}^{\lfloor (t-1)/2\rfloor}
\binom{t}{2k+1}z^{2k}
=
\frac{(1+z)^t-(1-z)^t}{2z}.
\]
For $z=0$, the right-hand side is understood as its limiting value, equal to $t$.
Moreover, by the mean value theorem applied to the function $s\mapsto s^t$
on the interval $[1-z,1+z]$, for $z>0$ and $t\geq1$ we have
\[
\frac{(1+z)^t-(1-z)^t}{2z}
\leq
t(1+z)^{t-1}.
\]
The same estimate is valid for $z=0$ by taking the limiting value.
Consequently, for $t\geq1$,
\[
\big\|\mathcal V_t^{a,b}\big\|_2
\leq
t\left(1+\sqrt{|b(q+1)-a|+2|b|\sqrt{q}}\right)^{t-1}.
\]

Moreover,  $\mathcal{W}_t^{a,b}$  and $\mathcal{V}_t^{a,b}$ are invariant under all automorphisms of $T_{q+1}$ and hence act as convolution operators with radial kernels.

\subsection{Wave kernels}
Let $\nu$ be the normalized radial measure concentrated on the set $\{x \in T_{q+1} : |x| = 1\}$, and let $\delta_0$ be the Dirac measure at zero. In the sequel, we will view $\delta_0$ and $\nu$ both as measures on $\mathbb{N}_0$ such that
$$
\delta_0(n)=\left\{
              \begin{array}{ll}
                1, & \text{  if  } n=0, \\
                0, & \text{otherwise},
              \end{array}
            \right. \quad \text{and} \quad
\nu(n)=\left\{
              \begin{array}{ll}
                \frac{1}{q+1}, & \text{  if  } n=1, \\
                0, & \text{otherwise},
              \end{array} \right.
$$
and also as radial functions on $T_{q+1}$ defined by $\delta_0(x):=\delta_0(d(x,x_0))=\delta_{x=x_0}$ and $\nu(x):=\nu(d(x,x_0))$. With a slight abuse of notation we will identify the vertex set $V(T_{q+1})$ with $T_{q+1}$.

The convolution of two functions $f, g : T_{q+1} \to \mathbb{R}$, where $f$ is a radial function on $T_{q+1}$ at a vertex $x$, is defined by
\begin{equation}\label{eq. def. conv}
(f \ast g)(x) = \sum_{y \in T_{q+1}} f(d(x,y)) \, g(y),
\end{equation}
where $d(x,y)$ denotes the graph distance between $x$ and $y$.
For real numbers $a,b$ with $b\neq 0$ and $x\in T_{q+1}$, we then have
\begin{align*}
(b\Delta_{q+1}-aI)f(x)&= b(q+1)\left(\delta_0 -\nu \right)\ast f(x)-af(x)\\&= (q+1)\left(\left(b-\frac{a}{q+1}\right) \delta_0 -b\nu \right)\ast f(x).
\end{align*}

We define the $k$-fold convolution $(B\delta_0-A\nu)^{(\ast k)}$ inductively, for any integer $k\geq 0$ and real numbers $A,B$ as follows:
\begin{itemize}
  \item for $k=0$ and $x\in T_{q+1}$, we set $(B\delta_0-A\nu)^{(\ast 0)}(x):=\delta_0(x)$;
  \item for $k=1$ and  $x\in T_{q+1}$, we define  $(B\delta_0-A\nu)^{(\ast 1)}(x):=(B\delta_0-A\nu)(x)$;
\item for $k\geq 2$ and $x\in T_{q+1}$
$$(B\delta_0-A\nu)^{(\ast k)}(x):=((B\delta_0-A\nu)^{(\ast (k-1))} \ast (B\delta_0-A\nu) )(x).$$
\end{itemize}

Since the wave kernels are radial functions, for a vertex $x\in T_{q+1}$ we denote by $|x|$ its distance from the root $x_0$ and write the wave kernels $W_{q+1}^{a,b}(x_0,x;t)$ and $V_{q+1}^{a,b}(x_0,x;t)$ on the tree $T_{q+1}$ simply as $W_{q+1}^{a,b}(|x|;t)$  and $V_{q+1}^{a,b}(|x|;t)$. We have the following lemma.
\begin{lemma} \label{lem: kernels on tree}
For any real parameters $a,\, b$, with $b\neq 0$, the functions
\begin{equation}\label{wave w as conv}
W_{q+1}^{a,b}(|x|;t) = \sum_{k=0}^{\infty} (-1)^k h_{2k}(t,0)\,(q+1)^k \left(\left(b-\frac{a}{q+1}\right) \delta_0 -b\nu \right)^{(*k)}(x),\,\, x\in T_{q+1},
\end{equation}
and
\begin{equation} \label{wave V as conv}
V_{q+1}^{a,b}(|x|;t) = \sum_{k=0}^{\infty} (-1)^k h_{2k+1}(t,0)\,(q+1)^k \left(\left(b-\frac{a}{q+1}\right) \delta_0 -b\nu \right)^{(*k)}(x),\,\, x\in T_{q+1},
\end{equation}
satisfy equation \eqref{wave_eq_disc} and the initial conditions \eqref{cond1} and \eqref{cond on V}, respectively.
\end{lemma}
\begin{proof}
First, note that the coefficients $h_{2k}(t,0)=\binom{t}{2k}$ satisfy the binomial difference identity
$$
\partial_t^2 h_{2k}(t,0)=h_{2k-2}(t,0)\quad(k\ge1), \qquad \partial_t^2 h_0(t,0)=0.
$$
For brevity, set $B=b-\frac{a}{q+1}$. Since the sums defining $W_{q+1}^{a,b}(|x|;t)$ and $V_{q+1}^{a,b}(|x|;t)$ are finite for fixed $t$, we may interchange the summation with the action of the operators $\partial_t$ and $\Delta_{q+1}$, which yields 
\begin{equation*}
\begin{split}
\partial_t^2 W_{q+1}^{a,b}(|x|;t)
&= \sum_{k=1}^{\infty}(-1)^k h_{2k-2}(t,0)(q+1)^k \left(B\delta_0 -b\nu \right)^{(*k)}(x)\\
&=-\sum_{k=0}^{\infty}(-1)^k h_{2k}(t,0)(q+1)^{k+1} \left( B \delta_0 -b\nu \right)^{(\ast(k+1))}(x)\\
&=
-(q+1)\left(B \delta_0 -b\nu \right)\ast \sum_{k=0}^{\infty} (-1)^k h_{2k}(t,0)\,(q+1)^k \left(B\delta_0 -b\nu \right)^{(*k)}(x) \label{eq. wave eq. prop}\\
&=
-\Delta_{q+1}^{a,b}W_{q+1}^{a,b}(|x|;t).
\end{split}
\end{equation*}
Hence, $W_{q+1}^{a,b}(|x|;t)$ satisfies \eqref{wave_eq_disc}.

\medskip
At $t=0$, we have $h_0(0,0)=1$ and $h_{\ell}(0,0)=0$ for all $\ell\ge1$, so
$$
W_{q+1}^{a,b}(|x|;0)=\left( B\delta_0 - b\nu \right)^{(*0)}(x)= \delta_{x=x_0}(x)
$$
and
$$
\partial_t W_{q+1}^{a,b}(|x|;0)=\sum_{k=1}^{\infty} (-1)^k h_{2k-1}(0,0)\,(q+1)^k \left( B\delta_0 - b\nu \right)^{(*k)}(x)\equiv 0.
$$

A similar argument shows that $V_{q+1}^{a,b}(|x|;t)$ satisfies \eqref{wave_eq_disc}. To show that the initial conditions \eqref{cond on V} are satisfied, we use that $h_{2k+1}(0,0)=0$ for all $k\geq 0$ to see that $V_{q+1}^{a,b}(|x|;0)=0$, while delta differentiation of $V_{q+1}^{a,b}$ with respect to $t$ yields that
$$
\partial_t V_{q+1}^{a,b}(|x|;0)=W_{q+1}^{a,b}(|x|;0)=\delta_0(x),
$$
which proves \eqref{cond on V}.
\end{proof}

\section{Integral representation of discrete Bessel functions}\label{sec:int-repr}

The discrete $I$-Bessel and $J$-Bessel functions, introduced in \cite{Cu15} and further generalized and studied in \cite{BC, Slav18, KS22, CHJSV, BSS24}, for $t,n\in\mathbb{N}_0$ (with $n\le t$) and $c\in\mathbb{C}$ can be expressed as 
\begin{align}
I^{c}_{n}(t)
&= \sum_{\ell=0}^{\lfloor (t-n)/2\rfloor}
\frac{t!}{\ell!\,(t-2\ell-n)!\,(n+\ell)!}\left(\frac{c}{2}\right)^{2\ell+n},\label{eq:def-I-Bessel}\\
J^{c}_{n}(t)
&= \sum_{\ell=0}^{\lfloor (t-n)/2\rfloor}
\frac{(-1)^{\ell}t!}{\ell!\,(t-2\ell-n)!\,(n+\ell)!}\,\left(\frac{c}{2}\right)^{2\ell+n}.\label{eq:def-J-Bessel}
\end{align}
Note that $I^{c}_{n}(t)= J^{c}_{n}(t)=0$ for all $c\in\mathbb{C}$ and $t,n\in\mathbb{N}_0$ such that $n> t$. 

In this section, we prove discrete analogues of the classical integral representations for the $I$-Bessel and $J$-Bessel functions, which will play an important role in the later analysis. Namely, we prove the following proposition.

\begin{proposition}\label{prop:IJ-int}
Let $n,k\in\mathbb{N}_0$ and let $\alpha\in \mathbb{R} \setminus \{0\}$. Then:
\begin{enumerate}[label=\textup{(\roman*)}]
\item \label{prop:IJ-int-i} $\displaystyle
\frac{1}{2\pi}\int_{-\pi}^{\pi} (1+\alpha\cos u)^{k}\,e^{inu}\,du
=
\begin{cases}
I_{n}^{\alpha}(k), & \text{if } k\ge n,\\[4pt]
0, & \text{otherwise,}
\end{cases}$

\item \label{prop:IJ-int-ii} $\displaystyle
\frac{1}{2\pi}\int_{-\pi}^{\pi} (1-i\alpha\sin u)^{k}\,e^{inu}\,du
=
\begin{cases}
J_{n}^{\alpha}(k), & \text{if } k\ge n,\\[4pt]
0, & \text{otherwise.}
\end{cases}$
\end{enumerate}
\end{proposition}

Before proving Proposition \ref{prop:IJ-int}, let us first explain why it can be viewed as the discrete-time analogue of the integral representations
$$
\begin{aligned}
\mathcal I_\nu(\alpha) &= \frac{1}{\pi}\int_{0}^{\pi} e^{\,\alpha\cos u}\,\cos(\nu u)\,du= \frac{1}{2\pi}\int_{-\pi}^{\pi} e^{\,\alpha\cos u}e^{\,i\nu u}\,du,\\
\mathcal J_\nu(\alpha) &= \frac{1}{\pi}\int_{0}^{\pi}\cos\!\big(\nu u - \alpha\sin u\big)\,du
= \frac{1}{2\pi}\int_{-\pi}^{\pi} e^{\,-i \alpha\sin u}e^{\,i\nu u}\,du,
\end{aligned}
$$
of the classical $I$-Bessel function $\mathcal I_\nu(\alpha)$ and $J$-Bessel function $\mathcal J_\nu(\alpha) $ (see \cite[8.431.5 and 8.411.1]{GR}), valid for $\nu\in\mathbb{N}_0$ and $z\in\mathbb{C}$.

Namely, in the integral representations of $\mathcal I_\nu(\alpha)$ and $\mathcal J_\nu(\alpha) $, the integrands contain exponential factors of the form $e^{\pm \alpha\cos u}$ and $e^{\pm i \alpha\sin u}$. The exponential function on the time scale $\mathbb{T}=\mathbb{N}_0$ is
$$
e_\beta(t,0)=(1+\beta)^{t}
$$
(see \cite{Cu15}, formulas (49)--(52) with $h=1$  for further details and equivalent forms). Therefore, $$(1+\alpha\cos u)^{k}=e_{\alpha \cos u}(k,0),$$ meaning that it is a discrete-time ``version'' of the function $e^{\,\alpha \cos u}$. Similarly, $$(1-i\alpha\sin u)^{k}= e_{-i\alpha \sin u}(k,0),$$ hence the function $(1-i\alpha\sin u)^{k}$ is a discrete-time ``version'' of the function $ e^{\,-i \alpha\sin u}$.

The proof of Proposition~\ref{prop:IJ-int} will rely on the following two auxiliary integral formulas, given in the next lemma.

\begin{lemma}\label{lem:integrali}
For $j,n\in\mathbb{N}_0$, the following holds:
\begin{enumerate}[label=\textup{(\alph*)}]
  \item $\displaystyle \frac{1}{2\pi}\!\int_{-\pi}^{\pi}\cos^{j}u\,e^{inu}\,du
  =
  \begin{cases}
    \displaystyle \frac{1}{2^{j}}\binom{j}{\frac{j-n}{2}}, & \text{if } j\ge n \text{ and } j-n \text{ is even},\\[6pt]
    0, & \text{otherwise},
  \end{cases}$

  \item $\displaystyle \frac{1}{2\pi}\!\int_{-\pi}^{\pi}\sin^{j}u\,e^{inu}\,du
  =
  \begin{cases}
    \displaystyle \frac{(-1)^{\frac{j+n}{2}}}{(2i)^{\,j}}\binom{j}{\frac{j-n}{2}}, & \text{if } j\ge n \text{ and } j-n \text{ is even},\\[6pt]
    0, & \text{otherwise}.
  \end{cases}$
\end{enumerate}
\end{lemma}

\begin{proof}[Proof of Lemma~\ref{lem:integrali}] We could not find an explicit reference for these formulas; therefore, we include a short proof for completeness, although the result is likely well known.

\textbf{(a)} Let
$$
I=\frac{1}{2\pi}\int_{-\pi}^{\pi}\cos^{j}u\,e^{inu}\,du.
$$
Writing the cosine function in terms of exponentials, we obtain
$$
I=\frac{1}{2\pi}\int_{-\pi}^{\pi}\Bigg[\frac{1}{2^{j}}
\sum_{k=0}^{j}\binom{j}{k}e^{\,i(j-2k)u}\Bigg]e^{inu}\,du
=\frac{1}{2^{j}}\sum_{k=0}^{j}\binom{j}{k}\,
\frac{1}{2\pi}\int_{-\pi}^{\pi}e^{\,i(n+j-2k)u}\,du.
$$
The orthogonality of exponentials implies that the only nonzero contribution in the integral above occurs when $k=\frac{j+n}{2}$, which requires
$j\ge n$ and $j-n$ even. Therefore,
\[
I=\frac{1}{2^{j}}\binom{j}{\frac{j+n}{2}}
=\frac{1}{2^{j}}\binom{j}{\frac{j-n}{2}},
\]
which proves part (a). Part (b) is deduced analogously and is therefore omitted.
\end{proof}

We now proceed with the proof of Proposition \ref{prop:IJ-int}.

\begin{proof}[Proof of Proposition~\ref{prop:IJ-int}]

\textbf{(i)} Let us introduce the notation
$$
M(n,k;\alpha)=\frac{1}{2\pi}\int_{-\pi}^{\pi}(1+\alpha\cos u)^{k}\,e^{inu}\,du.
$$
Applying the binomial formula, 
and using Lemma~\ref{lem:integrali}\,(a), we get
$$
M(n,k;\alpha)=\sum_{j=0}^{k}\binom{k}{j}\,\alpha^{j}\,\frac{1}{2^{j}}
\binom{j}{\frac{j-n}{2}},
\qquad \text{with $j\ge n$ and $j-n$ even.}
$$
We transform the above sum by introducing the substitution $j-n=2\ell$ with $\ell\in\mathbb{N}_0$.
Since $0\le j\le k$ and we keep only terms with $j\ge n$ and $j-n$ even,
it follows that $\ell=0,1,\dots,\left\lfloor\frac{k-n}{2}\right\rfloor$.
Then $j=n+2\ell$, hence using \eqref{eq:def-I-Bessel} we obtain
\begin{align*}
M(n,k;\alpha)&=\sum_{\ell=0}^{\left\lfloor\frac{k-n}{2}\right\rfloor}
\binom{k}{\,n+2\ell\,}\binom{n+2\ell}{\,\ell\,}
\left(\frac{\alpha}{2}\right)^{2\ell+n}
=
\sum_{\ell=0}^{\left\lfloor\frac{k-n}{2}\right\rfloor}
\frac{k!}{\ell!\,(k-2\ell-n)!\,(n+\ell)!}\,
\left(\frac{\alpha}{2}\right)^{2\ell+n}\\
&=I_{n}^{\alpha}(k).
\end{align*}

\textbf{(ii)} Let
$$
\tilde M(n,k;\alpha)=\frac{1}{2\pi}\int_{-\pi}^{\pi}(1-i\alpha\sin u)^{k}\,e^{inu}\,du.
$$
Proceeding analogously as in the proof of part (i), we obtain
$$
\tilde M(n,k;\alpha)=\sum_{j=0}^{k}\binom{k}{j}\,i^{\,j}(-1)^{j}\alpha^{j}\,
\frac{(-1)^{\frac{j+n}{2}}}{(2i)^{\,j}}
\binom{j}{\frac{j-n}{2}}
=\sum_{j=0}^{k}\binom{k}{j}\,\alpha^{j}\,
\frac{(-1)^{\frac{3j+n}{2}}}{2^{\,j}}
\binom{j}{\frac{j-n}{2}}.
$$
Introducing the substitution $j-n=2\ell$ with $\ell\in\mathbb{N}_0$, we have
$\ell=0,1,\dots,\left\lfloor\frac{k-n}{2}\right\rfloor$ and $j=n+2\ell$. Therefore,
\[
\tilde M(n,k;\alpha) =\sum_{\ell=0}^{\left\lfloor\frac{k-n}{2}\right\rfloor}
\frac{(-1)^{\,\ell}k!}{\ell!\,(k-2\ell-n)!\,(n+\ell)!}\,
\left(\frac{\alpha}{2}\right)^{2\ell+n}=\,J_{n}^{\alpha}(k),
\]
which completes the proof of the proposition.
\end{proof}

\section{Wave equation on the $(q+1)$-regular tree}\label{sec. wave on tree sol}

Let $a$ and $b$ be real numbers such that $b\neq 0$ and $b(q+1)-a\geq 2|b|\sqrt{q}$.  Let $T_{q+1}$ be a rooted tree with root $x_0$, and let  $r=|x|$ denote, as before, the distance from the root $x_0$.

In this section, for two given functions $w, v: T_{q+1}\to\mathbb{R}$ that belong to $L^2(T_{q+1})$, we study the following initial value problem: Find a function $w_{q+1}^{a,b}(x;t)$ that is a solution to  \eqref{wave_eq_disc} with $\mathbb{T}=\mathbb{N}_0$ satisfying the initial conditions
\begin{equation}\label{initial cond general}
w_{q+1}^{a,b}(x;0)=w(x),\quad \text{and}\quad \partial_t w_{q+1}^{a,b}(x;0)= v(x),\,\,\, x\in T_{q+1} .
\end{equation}

\subsection{Fundamental solutions to the wave equation}

To solve the general initial value problem on the tree, we first derive explicit formulas for the discrete space-time wave kernels, meaning that we prove Theorem \ref{Tqkernel_a}.

We assume $q\geq 1$ to be fixed and recall the notation
$$c_{a,b} = \dfrac{2b\sqrt{q}}{b(q+1) - a}, \quad\, d_{a,b}=b(q+1)-a.$$
Equation \eqref{wave_eq_disc}, with $\mathbb{T}=\mathbb{N}_0$, is equivalent to the following radial equation:
\begin{equation}\label{eq. wave on tree}
\partial_t^2 W(r;t) =
\begin{cases}
-\big(b(q+1)-a\big)\,W(0;t) + b(q+1)\,W(1;t), & \text{if } r = 0, \\[6pt]
-\big(b(q+1)-a\big)\,W(r;t) + bq\,W(r+1;t) + bW(r-1;t), & \text{if } r \geq 1.
\end{cases}
\end{equation}

\begin{proof}[Proof of Theorem~\ref{Tqkernel_a}] 
We prove \eqref{wave_kernel_a}; the proof of \eqref{wave_kernel_b} follows from the fact that
$V_{q+1}^{a,b}(r;t)$ is the integral of the solution $W_{q+1}^{a,b}(r;t)$ from zero to $t$ in the time scale $\mathbb{T}=\mathbb{N}_0$, meaning that for $t\geq 1$ one has
\begin{equation}\label{eq. V as int of W}
    V_{q+1}^{a,b}(r;t)=\sum_{m=0}^{t-1}W_{q+1}^{a,b}(r;m).
\end{equation}
We will prove \eqref{wave_kernel_a} by proving that expressions  \eqref{wave w as conv} and \eqref{wave_kernel_a} are equal for $x\in T_{q+1}$ with $|x|=r$. 

The case $q=1$ is included in our notation, but it requires a separate argument in the proof. In this case the
$(q+1)$-regular tree is the $2$-regular tree $T_2$, which is identified
with the integer lattice $\mathbb Z$. Thus, for $q\geq2$ we use the spherical Fourier transform on homogeneous trees, while for $q=1$ we use the classical Fourier transform on
$\mathbb Z$.

Assume first that $q\ge2$. In this case, we may use the spherical Fourier transform on the tree. The convolution kernel of the operator $\Delta^{a,b}_{T_{q+1}}=b\Delta_{q+1}-aI$ is
$$
g_{a,b}(x):=\big((b(q+1)-a)\delta_0 - b(q+1)\nu\big)(x)
=\begin{cases}
b(q+1)-a, & |x|=0,\\
-b, & |x|=1,\\
0, & \text{otherwise}.
\end{cases}
$$
Therefore,  the spherical Fourier transform of the convolution kernel $g_{a,b}$ is given by
\begin{equation}\label{ft}
\begin{split}
\tilde{g}_{a,b}\left(z\right)&= \sum_{x \in T_{q+1}}g_{a,b}\left(x\right) \cdot \phi_z(x)=(b(q+1)-a)\phi_z(0)- b(q+1)\phi_z(1)\\
&= b(q+1)-a - b(q+1)\left(\boldsymbol{c}(z) q^{iz -\frac{1}{2}} +\boldsymbol{c}(-z) q^{-iz - \frac{1}{2}}\right),
\end{split}
\end{equation}
where spherical harmonics $\phi_z(x)$ (meaning the radial eigenfunctions of the Laplace operator, satisfying the normalization condition $\phi_z(0)=1$) are given by
\[\phi_z(x) =\boldsymbol{c}(z) q^{(iz - \frac{1}{2})|x|} +\boldsymbol{c}(-z) q^{(-iz - \frac{1}{2})|x|}, \  \forall z \in \mathbb{C} \setminus (\tau/2) \mathbb{Z},
\]
where $\tau=\frac{2\pi}{\log{q}}$ and $\boldsymbol{c}$ is the meromorphic function defined by
$$\boldsymbol{c}(z) = \frac{\sqrt{q}}{q+1} \frac{q^{\frac{1}{2} + iz} - q^{-\frac{1}{2} - iz}}{q^{iz} - q^{-iz}}, \quad \forall z \in \mathbb{C} \setminus (\tau/2) \mathbb{Z}$$
(see Section 2 of \cite{MS99} for details).

Since
\begin{equation*}
\boldsymbol{c}(z) q^{iz -\frac{1}{2}}=\frac{1}{q+1} \cdot \frac{q^{\frac{1}{2}+2iz}-q^{-\frac{1}{2}}}{q^{iz}-q^{-iz}} \quad \text{and}  \quad \boldsymbol{c}(-z) q^{-iz -\frac{1}{2}}=\frac{1}{q+1} \cdot \frac{q^{\frac{1}{2}-2iz}-q^{-\frac{1}{2}}}{q^{-iz}-q^{iz}},
\end{equation*}
from \eqref{ft}, a simple calculation yields that
\begin{equation} \label{conv_kernel_transform}
\tilde{g}_{a,b}\left(z\right)=d_{a,b}\Big[\,1- c_{a,b}\cos\!\big(z\log{q}\big)\,\Big].
\end{equation}

Now, using \eqref{ft} and \eqref{conv_kernel_transform}, we obtain that the spherical Fourier transform of the wave kernel is given by
\begin{equation*}
\tilde{W}^{a,b}_{q+1,t}\left(z\right)
= \sum_{k=0}^{\infty}(-1)^k \, h_{2k}(t,0)\, d_{a,b}^k\big(1- c_{a,b}\cos(z\log q)\big)^{k},
\end{equation*}
where the sum on the right-hand side is finite, because $h_{2k}(t,0)=0$ for $k>t/2$. The spherical inversion yields that
\begin{equation}\label{inverse1}
    \begin{split}
W^{a,b}_{q+1}\left(|x|;t\right)&=2c_G \int_{-\frac{\pi}{\log{q}}}^{\frac{\pi}{\log{q}}} \tilde{W}^{a,b}_{q+1,t}\left(z\right) \left(\boldsymbol{c}(z)\right)^{-1} q^{(-iz - \frac{1}{2})|x|} \,dz \\
&=2c_G\sum_{k=0}^{\lfloor \frac{t}{2} \rfloor} (-1)^k h_{2k}\left(t,0\right)\,d_{a,b}^k
\int_{-\frac{\pi}{\log{q}}}^{\frac{\pi}{\log{q}}}\!\Big(1- c_{a,b}\cos(z\log q)\Big)^{k} \left(\boldsymbol{c}(z)\right)^{-1}  q^{(-iz - \frac{1}{2})|x|} \,dz ,
    \end{split}
\end{equation}
where $c_G=\frac{q\log{q}}{4\pi(q+1)}$.

Since
\begin{equation*}
   \left(\boldsymbol{c}(z)\right)^{-1}=\frac{q+1}{\sqrt{q}} \cdot \frac{q^{iz}-q^{-iz}}{q^{\frac{1}{2} + iz} - q^{-\frac{1}{2} - iz}}=
\frac{q+1}{q}  \cdot \frac{1-e^{-2iz\log{q}}}{1-\frac{1}{q}e^{-2iz\log{q}}},
\end{equation*}
using $c_G=\frac{q\log{q}}{4\pi(q+1)}$ and the substitution $u = -z\log{q}$, equation~\eqref{inverse1} becomes
\begin{equation}\label{w_t_1}
W^{a,b}_{q+1}\left(|x|;t\right)
=\frac{1}{2\pi}q^{-\frac{|x|}{2}}\sum_{k=0}^{\lfloor \frac{t}{2} \rfloor} (-1)^k h_{2k}\left(t,0\right)\,d_{a,b}^k
\int_{-\pi}^{\pi} \Big(1- c_{a,b}\cos u\Big)^{k} \frac{e^{iu|x|}-e^{iu(|x|+2)}}{1-\frac{1}{q}e^{2iu}} \,du.
\end{equation}
Let us denote
$$I=\frac{1}{2\pi}\int_{-\pi}^{\pi} \Big(1- c_{a,b}\cos u\Big)^{k} \frac{e^{iu|x|}-e^{iu(|x|+2)}}{1-\frac{1}{q}e^{2iu}} \,du.$$
Since
$$
\frac{1}{1 - \frac{1}{q} e^{2iu}} = \sum_{\ell=0}^{\infty} \left(\frac{1}{q} e^{2iu}\right)^\ell = \sum_{\ell=0}^{\infty} \frac{e^{2i \ell u}}{q^\ell},
$$
from the uniform convergence of the above series, we get
\begin{equation}\label{integ_I}
I=\sum_{\ell=0}^{\infty} q^{-\ell} \frac{1}{2\pi}\int_{-\pi}^{\pi} \Big(1- c_{a,b}\cos u\Big)^{k} \left(e^{iu(|x|+2\ell)} - e^{iu(|x|+2\ell+2)} \right) \,du.
\end{equation}
Applying Proposition~\ref{prop:IJ-int}\,(i) with $\alpha=-c_{a,b}\neq 0$, $n=|x|+2\ell$ and $n=|x|+2\ell+2$, we get
\begin{equation}\label{integ_I2}
I
= \sum_{\ell=0}^{\lfloor\frac{k-|x|}{2}\rfloor} q^{-\ell} \left(I_{|x|+2\ell}^{-c_{a,b}}(k)-I_{|x|+2\ell+2}^{-c_{a,b}}(k)\right).
\end{equation}

Substituting \eqref{integ_I2} into \eqref{w_t_1} yields that
\begin{equation*}
\begin{split}
W_{q+1}^{a,b}(|x|;t)
&= \sum_{k=0}^{\lfloor \frac{t}{2} \rfloor} (-1)^{k} h_{2k}(t,0)\,d_{a,b}^k\,q^{-\frac{|x|}{2}}
\sum_{\ell=0}^{\lfloor \frac{k - |x|}{2} \rfloor} q^{-\ell} \left( I_{|x|+2\ell}^{-c_{a,b}}(k) - I_{|x|+2\ell+2}^{-c_{a,b}}(k) \right)\\
&= \sum_{k=0}^{\lfloor \frac{t}{2} \rfloor} (-1)^{k} h_{2k}(t,0)\,d_{a,b}^k \left( q^{-\frac{|x|}{2}} I_{|x|}^{-c_{a,b}}(k) - (q-1) \sum_{\ell=1}^{\lfloor \frac{k - |x|}{2} \rfloor} q^{-\frac{|x|+2\ell}{2}} I_{|x|+2\ell}^{-c_{a,b}}(k) \right).
\end{split}
\end{equation*}
Indeed, the second equality follows by using the weighted telescoping identity 
\[ \sum_{\ell=0}^{N}q^{-\ell}(A_\ell-A_{\ell+1}) = A_0-(q-1)\sum_{\ell=1}^{N}q^{-\ell}A_\ell-q^{-N}A_{N+1},
\]
with $N=\lfloor (k-|x|)/2\rfloor$ and $A_\ell=I_{|x|+2\ell}^{-c_{a,b}}(k)$. Here the last term vanishes because $A_{N+1}=I_{|x|+2N+2}^{-c_{a,b}}(k)=0$, since $|x|+2N+2>k$.

By setting $r=|x|$ and using the identity $I_n^{-c}(k)=(-1)^nI_n^c(k)$, we get \eqref{wave_kernel_a}.

When $q=1$, the $(q+1)$-regular tree $T_2$ is isomorphic to the line graph $\mathbb Z$. In this case, the convolution kernel of the operator $\Delta^{a,b}_{T_{2}}=b\Delta_{T_{2}}-aI$ is 
\[ g_{a,b}(n):=\big((2b-a)\delta_0 - 2b\nu\big)(n) =\begin{cases} 2b-a, & n=0,\\ -b, & |n|=1,\\ 0, & \text{otherwise}. \end{cases} 
\] 
We use the classical Fourier transform on $\mathbb Z$. Therefore, the Fourier transform of $g_{a,b}$ is 
\begin{equation*}
\begin{split} 
\widehat{g}_{a,b}(z) &= \sum_{n \in \mathbb{Z}} g_{a,b}(n)\, e^{-inz} =(2b-a)-be^{-iz}-be^{iz}\\ &=2b-a-2b\cos z =d_{a,b}\Bigl(1-c_{a,b}\cos z\Bigr). 
\end{split} 
\end{equation*} 
Then, the Fourier transform of the wave kernel is given by 
\begin{equation*}
\widehat{W}^{a,b}_{2,t}(z) = \sum_{k=0}^{\lfloor \frac{t}{2} \rfloor}(-1)^k \, h_{2k}(t,0)\, d_{a,b}^k\bigl(1- c_{a,b}\cos z\bigr)^{k}. \end{equation*} 
Since the sum on the right-hand side is finite, applying the Fourier inversion formula and Proposition~\ref{prop:IJ-int}\,(i) yields
\begin{equation}\label{eq. W2 evaluation}
\begin{split} 
W^{a,b}_{2}(|x|;t) &=\frac{1}{2\pi} \int_{-\pi}^{\pi} \widehat{W}^{a,b}_{2,t}(z)\, e^{i|x|z}\,dz\\ &=\frac{1}{2\pi}\sum_{k=0}^{\lfloor t/2 \rfloor} (-1)^k h_{2k}(t,0)\,d_{a,b}^k \int_{-\pi}^{\pi}\!\bigl(1- c_{a,b}\cos z\bigr)^{k}e^{i|x|z}\,dz\\ &=\sum_{k=0}^{\lfloor t/2 \rfloor} (-1)^k h_{2k}(t,0)\,d_{a,b}^k\, I_{|x|}^{-c_{a,b}}(k),
\end{split} 
\end{equation}
which coincides with \eqref{wave_kernel_a} for $q=1$, since $q-1=0$.
\end{proof}

\subsection{General initial value problem on the tree ${T}_{q+1}$}

We can now prove the following corollary, in which the solution to the general wave equation on the tree is expressed in terms of the wave kernels $W_{q+1}^{a,b}$ and $V_{q+1}^{a,b}$.

\begin{corollary}
  With the notation as above, the solution $w_{q+1}^{a,b}$ of the wave equation \eqref{eq. wave on tree} satisfying the initial condition \eqref{initial cond general}, at time $t\in\mathbb N_0$, is given by
$$
w_{q+1}^{a,b}(x;t)=\sum_{r=0}^{\lfloor\frac{t}{2}\rfloor}\sum_{y\in T_{q+1}:\, d(x,y)=r}\left(W_{q+1}^{a,b}(r;t)w(y) + V_{q+1}^{a,b}(r;t)v(y)\right)
$$
for $x\in T_{q+1}$.
\end{corollary}
\begin{proof}
The functions $W_{q+1}^{a,b}$ and $V_{q+1}^{a,b}$ are convolution kernels for the wave equation \eqref{eq. wave on tree} with initial conditions \eqref{cond1} and \eqref{cond on V}, respectively. Hence, the solution to \eqref{eq. wave on tree} with the initial condition \eqref{initial cond general} is given by
$$
w_{q+1}^{a,b}(x;t)=(W_{q+1}^{a,b}\ast w )(x) +(V_{q+1}^{a,b}\ast v )(x). 
$$
The functions $W_{q+1}^{a,b}$ and $V_{q+1}^{a,b}$ are radial and such that $W_{q+1}^{a,b}(r;t)=V_{q+1}^{a,b}(r;t)=0$ for $r>\lfloor\frac{t}{2}\rfloor$. Therefore, by the definition \eqref{eq. def. conv} of convolution, the stated formula follows immediately. 
\end{proof}

\section{Boundary case and applications}\label{sec. applic on tree}

In this section, we study the special case of the (shifted) combinatorial Laplacian, meaning that we assume $b>0$ and take the maximal value $a=b(q+1-2\sqrt{q})$. We refer to this setting as the boundary setting. In the proposition \ref{prop:boundary} below, we deduce the expression for the wave kernels in terms of the $J$-Bessel functions.
Then, we apply the uniqueness of the wave kernel on the $2$-tree $\mathbb{Z}=T_2$ to derive an identity relating discrete $I$-Bessel and $J$-Bessel functions.

\begin{proposition}\label{prop:boundary}
The wave kernels associated with the shifted combinatorial Laplacian $b\Delta_{q+1}-b(q+1-2\sqrt q)I$, $b>0$, satisfying the initial conditions \eqref{cond1} and \eqref{cond on V}, for $x\in T_{q+1}$ with $|x|=r$, are given by
\begin{equation}\label{eq:J-form}
W^{b(q+1-2\sqrt{q}),b}_{q+1}(r;t)
=
q^{-\tfrac{r}{2}}\,J_{2r}^{\,c_J(b)}(t)
-(q-1)\sum_{\ell=1}^{\lfloor \tfrac{t-2r}{4}\rfloor}q^{-\tfrac{r+2\ell}{2}}\,J_{2r+4\ell}^{\,c_J(b)}(t),
\end{equation}
and
\begin{equation}\label{eq:J-form2}
V^{b(q+1-2\sqrt{q}),b}_{q+1}(r;t)
= \sum_{m=0}^{t-1}\left[
q^{-\tfrac{r}{2}}\,J_{2r}^{\,c_J(b)}(m)
-(q-1)\sum_{\ell=1}^{\lfloor \tfrac{t-2r}{4}\rfloor}q^{-\tfrac{r+2\ell}{2}}\,J_{2r+4\ell}^{\,c_J(b)}(m)\right],
\end{equation}
for $t\geq 1$,\footnote{Note that $V^{b(q+1-2\sqrt{q}),b}_{q+1}(r;0)=0$.} where $J_{n}^{c_J(b)}(t)$ denotes the discrete $J$-Bessel function and $c_J(b)=2\,(qb^2)^{1/4}$.
\end{proposition}

\begin{proof} Assume first that $q\geq 2$. 
When $a=b(q+1-2\sqrt{q})$, we have $c_{a,b}=1$ and $d_{a,b}=2b\sqrt{q}$. The starting point for our proof is the computation of the integral \eqref{integ_I} appearing in formula \eqref{w_t_1}. We have that in this case
$$
I= \sum_{\ell=0}^{\infty} q^{-\ell} \frac{1}{2\pi}2^{k} \int_{-\pi}^{\pi}\sin^{2k}(u/2)  \left(e^{iu(r+2\ell)} - e^{iu(r+2\ell+2)} \right) \,du.
$$
For any integer $n\geq 0$, after the change of variables $u=\pi-2x$, we obtain
\begin{equation}\label{eq. sin int}
\frac{1}{2\pi}2^{k} \int_{-\pi}^{\pi}\sin^{2k}(u/2) e^{iun} \,du=(-1)^n \frac{2^{k+1}}{\pi}\int_{0}^{\pi/2}\cos^{2k}(x)\cos(2nx)\, dx.
\end{equation}
Applying \cite{GR}, formula 3.631.9 with $a=2n$ and $\nu=2k+1$, yields
\begin{equation}\label{eq. sin int eval}
\frac{2^{k+1}}{\pi}\int_{0}^{\pi/2}\cos^{2k}(x)\cos(2nx)\,dx
=
2^{-k}\left\{
\begin{array}{ll}
0, & \text{if } n\geq k+1,\\[4pt]
\binom{2k}{k-n}, & \text{if } 0\leq n\leq k.
\end{array}
\right.
\end{equation}
Thus, for $k\geq r$ we have
$$
I=2^{-k}\left[(-1)^r\binom{2k}{k-r}-(q-1)\sum_{\ell=1}^{\lfloor\frac{k-r}{2}\rfloor}q^{-\ell}(-1)^r
\binom{2k}{k-(r+2\ell)}\right],
$$
while $I=0$ for $k<r$.
Substituting this expression into formula \eqref{w_t_1} yields
\begin{multline}\label{eq. W precomp}
W_{q+1}^{b(q+1-2\sqrt{q}),b}(r;t)
=
\sum_{k=r}^{\left\lfloor \frac{t}{2}\right\rfloor}
(-1)^{k+r}\,
q^{-\frac{r}{2}+\frac{k}{2}}b^k\,
\binom{t}{2k}
\left(
\binom{2k}{k-r}
-
(q-1)\sum_{\ell=1}^{\left\lfloor \frac{k-r}{2}\right\rfloor}
q^{-\ell}\,
\binom{2k}{k-(r+2\ell)}
\right),
\end{multline}
where the empty sum is assumed to be zero.

Now, we prove the identity
\begin{equation}\label{eq. identity for J bess}
\sum_{k=r}^{\lfloor\frac{t}{2}\rfloor}(-1)^{k+r} q^{-\frac{r}{2}+\frac{k}{2}}b^k\binom{t}{2k}\sum_{\ell=0}^{\lfloor\frac{k-r}{2}\rfloor}q^{-\ell}\binom{2k}{k-(r+2\ell)} =\sum_{\ell=0}^{\lfloor\frac{t-2r}{4}\rfloor}q^{-\frac{r+2\ell}{2}} J_{2(r+2\ell)}^{c_J(b)}(t).
\end{equation}
We start with the change of variable $k=j+r$ and then interchange sums over $\ell$ and $j$ to obtain that the left-hand side of \eqref{eq. identity for J bess} equals
$$
\sum_{\ell=0}^{\lfloor\frac{t-2r}{4}\rfloor}\sum_{j=2\ell}^{\lfloor\frac{t-2r}{2}\rfloor}(-1)^jq^{\tfrac{j}{2}-\ell}b^{j+r}
\frac{t!}{(t-2j-2r)!(j-2\ell)!(j+2r+2\ell)!}.
$$
Finally, by the change of variables $j=m+2\ell$ in the second sum above, we see that the left-hand side of \eqref{eq. identity for J bess} equals
$$
\sum_{\ell=0}^{\lfloor\frac{t-2r}{4}\rfloor}\sum_{m=0}^{\lfloor\frac{t-2r-4\ell}{2}\rfloor}(-1)^m 
\frac{q^{\tfrac{m}{2}}b^{m+2\ell+r} \cdot t!}{m!(t-2m-(2r+4\ell))!(m+2r+4\ell)!}=\sum_{\ell=0}^{\lfloor\frac{t-2r}{4}\rfloor}q^{-\frac{r+2\ell}{2}} J_{2(r+2\ell)}^{c_J(b)}(t),
$$
as claimed (the second equality follows from the definition \eqref{eq:def-J-Bessel} of the discrete J-Bessel function).
The identity \eqref{eq. identity for J bess} combined with \eqref{eq. W precomp} yields the statement \eqref{eq:J-form}. 

When $q=1$, equation \eqref{eq. W2 evaluation} with $c_{a,b}=1$, $d_{a,b}=2b$, combined with  \eqref{eq. sin int} and \eqref{eq. sin int eval}, gives
$$
W_2^{0,b}=\sum_{k=r}^{\lfloor\frac{t}{2}\rfloor}(-1)^{k+r}b^k \binom{t}{2k} \binom{2k}{k-r}=\sum_{\ell=0}^{\lfloor\frac{t-2r}{2}\rfloor}(-1)^{\ell} \frac{(\sqrt{b})^{2\ell+2r} \cdot t!}{\ell!(t-2\ell-2r)!(\ell+2r)!}=J_{2r}^{2\sqrt{b}}(t),
$$
where we used the substitution $\ell=k-r$. This proves \eqref{eq:J-form} for $q=1$.

Finally,  \eqref{eq:J-form2} follows from \eqref{eq:J-form} and  \eqref{eq. V as int of W}.
\end{proof}

\begin{remark}
Proposition \ref{prop:boundary} can be proved differently, namely by starting from Theorem~\ref{Tqkernel_a} with $a=b(q+1-2\sqrt{q})$, $b>0$. The key observation is that, for $0\leq n\leq k$,
\begin{equation}\label{eq. indentity I1}
I_n^{\,1}(k)=2^{-k}\binom{2k}{\,k-n\,}.
\end{equation}
Indeed, starting from \eqref{eq:def-I-Bessel} and using the multinomial expansion, it follows that $I_n^c(k)$ is the coefficient of $z^n$ in 
\[ \left(1+\frac{c}{2}z+\frac{c}{2}z^{-1}\right)^k. \] 
For $c=1$, this becomes $2^{-k}z^{-k}(1+z)^{2k}$.
Hence the coefficient of $z^n$ is equal to the coefficient of
$z^{k+n}$ in $2^{-k}(1+z)^{2k}$. Therefore,
\[
I_n^{\,1}(k)
=
2^{-k}\binom{2k}{k+n}.
\]
Using the symmetry of binomial coefficients $\binom{2k}{k+n}=\binom{2k}{k-n}$, we obtain \eqref{eq. indentity I1}. The rest of the proof then proceeds as above.\footnote{We are thankful to the anonymous referee for suggesting the proof of \eqref{eq. indentity I1}}
\end{remark}

\begin{remark}
When $q=1$, the tree $T_2$ can be identified with the integer graph $\mathbb{Z}$. Fundamental solutions to the discrete wave equation \eqref{wave_eq_disc} with $b=c^2$, $c>0$, and $a=0$ on $\mathbb{Z}$ were derived in \cite[Theorem 3.1]{Slav18}, where it was shown that $W_{T_2}^{0,c^2}(r;t) = J_{2r}^{2c}(t)$, a result which coincides with \eqref{eq:J-form} for $q=1$.
\end{remark}

Using the uniqueness of the solution to the wave equation on $T_2=\mathbb{Z}$, see \cite[Theorem 2.3]{Slav17}, we obtain the following corollary.

\begin{corollary} \label{cor: I and J rel}
 With the notation as above, for any $c>0$ and any $n,t\in\mathbb N_0$, the identity \eqref{discrete_analog} holds true.
\end{corollary}
\begin{proof}
According to  \eqref{wave_kernel_a}, the right-hand side of \eqref{discrete_analog} is the wave kernel associated with the generalized Laplacian $\Delta_{T_2}^{0,c^2}$ satisfying the initial condition \eqref{cond1}. From \cite[Theorem 3.1]{Slav18} we have that the left-hand side of \eqref{discrete_analog} is the wave kernel associated with the generalized Laplacian $\Delta_{T_2}^{0,c^2}$ with the initial condition \eqref{cond1}. The result follows from the uniqueness of the wave kernel and the identity $I_n^{-c}(k)=(-1)^n I_n^c(k)$.
\end{proof}

\begin{remark}
Formula \eqref{discrete_analog} can be viewed as a discrete analogue of the following identity for classical Bessel functions given in \cite[Eq.~9.6.52]{AS}
\begin{equation*}
       \mathcal{J}_{\nu}(z) = \sum_{k=0}^{\infty} \frac{(-1)^k z^k}{k!} \mathcal{I}_{\nu+k}(z).
\end{equation*}

In the discrete setting, the function $J_{2n}^{2c}(t)$ corresponds to the fundamental solution, while the coefficients $(2c^2)^kI_n^{1}(k)$ play the role analogous to that of $I_{\nu+k}(z)$ in the continuous case. The discrete monomials $ h_{2k}(t,0)$ serve as counterparts to the continuous monomials $\dfrac{z^k}{k!}$, capturing the propagation structure in discrete time. 

This analogy highlights the structural similarity between continuous and discrete convolution-type identities involving Bessel-type functions.
\end{remark}

\section{Asymptotic behavior of the wave kernels}\label{sec. asymptotics}

In this section, we study the asymptotic behavior of the wave kernel $W_{q+1}^{a,b}(r;t)$ as a function of the radial distance $r$, as $r\to\infty$ (for fixed time $t$) and as a function of time $t$, as $t\to\infty$ (for fixed radial distance $r$). We also investigate the asymptotic behavior of the boundary case wave kernel when $q\to\infty$.

\subsection{Asymptotic behavior for large radial distance}

When the time variable is fixed, it is rather straightforward to deduce the following corollary.

\begin{corollary}
    For a fixed time $t\in\mathbb{N}_0$, we have
    $$ \lim_{r\to\infty}W_{q+1}^{a,b}(r;t)=0\quad\,\, \text{and}\quad\,\,  \lim_{r\to\infty}V_{q+1}^{a,b}(r;t)=0.
    $$
\end{corollary}
\begin{proof}
The statement follows from explicit expressions \eqref{wave_kernel_a} and \eqref{wave_kernel_b},  combined with the fact that $I_r^c(k)=0$ for all $k<r$.
\end{proof}

\subsection{Large time asymptotic behavior: boundary case $a=b(q+1-2\sqrt{q})$}

When time tends to infinity, the wave kernel is expected to oscillate. For $q=1$, it is proved in \cite{BSS24} that the wave kernel $W_2^{0,c^2}(r;t)$ oscillates as $t\to\infty$; however, the amplitude of these oscillations exhibits exponential growth. This is somewhat unexpected; yet, as seen in \cite{Ch22}, in some situations solutions to a discrete semilinear wave equation can blow up in finite time.

When $q\geq 2$, we describe the asymptotic behavior of the wave kernel $W_{q+1}^{b(q+1-2\sqrt{q}),b}(r;t)$ (with $b>0$) as $t\to\infty$, using results from \cite{BSS24}. We start with the following lemma.

\begin{lemma}
    For any $n,t\in\mathbb{N}_0$ and $c\in(0,\infty)$, we have, as $t\to\infty$,
\begin{multline}\label{eq. J difference}
J_{2n}^c(t) - J_{2n+4}^c(t)=\frac{(-1)^n}{\sqrt{\pi t\sin\theta}}\left( 1+c^2
\right) ^{\frac{t}{2}}\cos \left( \left( t+\frac{1}{2}\right)
\theta -\frac{\pi}{4} \right)\frac{(c/2)_{2n}t^{2n}}{(c/2)^{2n}(t+1)_{2n}}\times\\
\times \left[1-\frac{t^4(2n-c/2)(2n+1-c/2)(2n+2-c/2)(2n+3-c/2) }{(t+2n+1)(t+2n+2)(t+2n+3)(t+2n+4) (c/2)^4} + O\left(\frac{1}{t^{3/2}}\right)\right],
\end{multline}
where $\theta$ is defined by $\cos\theta=(1+c^2)^{-1/2}$, and $(a)_m:=a(a+1)\ldots (a+m-1)$ ($m\in\mathbb{N}_0$) is the Pochhammer symbol.
\end{lemma}
\begin{proof}
We recall the expression for the discrete $J$-Bessel function in terms of the Legendre function $P_\nu^\mu$ from \cite[Proof of Theorem 2]{BSS24}:
\begin{align*}
J_{2m}^c(t)&=\frac{(-c/2)_{2m} (-t)_{2m} (1+c^2)^{t/2}}{(c/2)^{2m}}\frac{(t-2m)!}{(t+2m)!}P_t^{2m}(\cos\theta) \\ &= \frac{(-c/2)_{2m} (1+c^2)^{t/2}}{(c/2)^{2m}}\frac{t^{2m}}{(t+1)_{2m}} t^{-2m}P_t^{2m}(\cos\theta).
\end{align*}
Combining this with $m=n$ and $m=n+2$, together with the asymptotic expansion \cite[formula 8.721.4]{GR} with $\nu=t$, $\mu=2n$, and $\mu=2n+4$, yields \eqref{eq. J difference}.
\end{proof}

In the following corollary, we show that, for $q\geq 2$, the wave kernel $W_{q+1}^{b(q+1-2\sqrt{q}),b}(r;t)$ oscillates in time $t$ with amplitude growing as $\left( 1+4b\sqrt{q}\right) ^{t/2}$. Figures~\ref{fig:q2} and~\ref{fig:q3} illustrate this behavior for the boundary wave kernels $W_3^{3-2\sqrt2,1}$ and $W_4^{4-2\sqrt3,1}$, respectively. They show pronounced sign oscillations together with rapid exponential amplification, while the magnitude of the oscillations decreases as the radial distance $r$ increases due to the exponentially decaying factor $q^{-r/2}$ in the asymptotic formula. Selected numerical values for the $T_3$ case are given in Table~\ref{tab:WRBoundary_q2} in the Appendix. Due to the fast growth of the wave kernels, most of the figures below, including Figures~\ref{fig:q2} and~\ref{fig:q3}, display the signed logarithmic transform defined by $\operatorname{sign}(W)\,|\log|W||$, if $W\neq 0$, and $0$, if $W=0$, of the corresponding numerical data.

\begin{figure}[h!]
    \centering
    \includegraphics[width=0.7\linewidth]{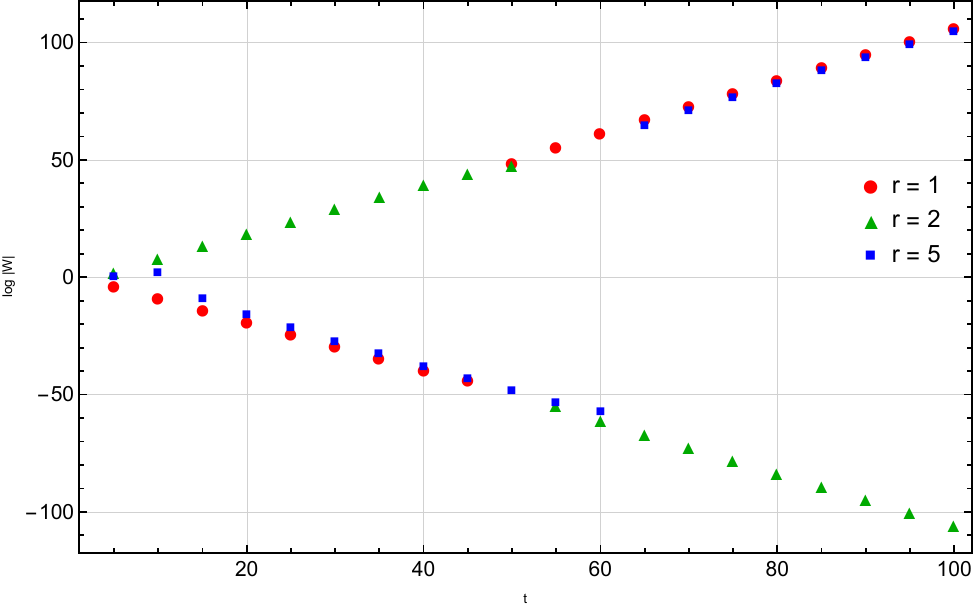}
    \caption{Signed logarithmic transform of $W_3^{3-2\sqrt{2},1}(r;t)$ for $r=1,2,5$.}
    \label{fig:q2}
\end{figure}

\begin{figure}[h!]
    \centering
    \includegraphics[width=0.7\linewidth]{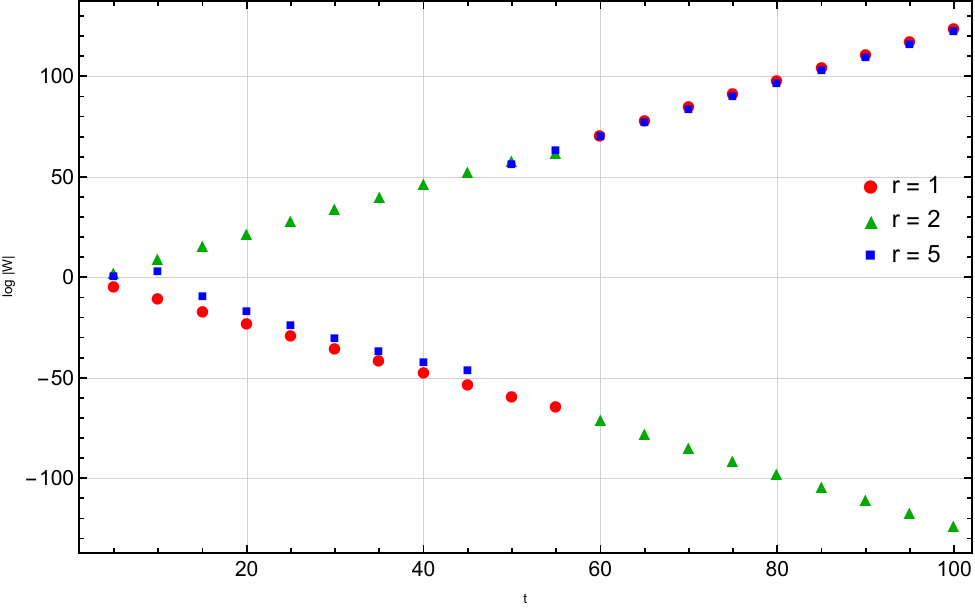}
    \caption{Signed logarithmic transform of $W_4^{4-2\sqrt{3},1}(r;t)$ for $r=1,2,5$.}
    \label{fig:q3}
\end{figure}

\begin{corollary} Let $q\geq 2$.
    For a fixed radial distance $r$, the wave kernel $W_{q+1}^{b(q+1-2\sqrt{q}),b}(r;t)$ oscillates in time $t$ with exponentially growing amplitude. In other words, for $\theta\in(0,\pi)$ such that $\cos\theta=(1+4b\sqrt{q})^{-1/2}$ we have
    $$
W_{q+1}^{b(q+1-2\sqrt{q}),b}(r;t)\sim \frac{(-1)^rq^{-r/2}}{\sqrt{\pi t\sin\theta}} C_b(r;t)\left( 1+4b\sqrt{q}
\right) ^{\frac{t}{2}}\cos \left( \left( t+\frac{1}{2}\right)
\theta -\frac{ \pi}{4} \right), \text{ \ as }%
t\rightarrow \infty \text{,}
$$
where $C_b(r;t)$ is an algebraic function of $b,\, r,\, t,\, q$ such that $\lim_{t\to\infty}|C_b(r;t)|\geq C_b(r)$, and $C_b(r)$ is a positive constant depending only on $b,\, r,\, q $.
\end{corollary}

\begin{proof} Our starting point is \eqref{eq:J-form}. We write it in the following form:
\begin{equation}\label{eq. W different form} 
W_{q+1}^{b(q+1-2\sqrt{q}),b}(r;t)=q^{-r/2}\left(\sum_{\ell=0}^{\lfloor\frac{t-2r}{4}\rfloor -1} q^{-\ell}\left(J_{2r+4\ell}^{c_J(b)}(t)-J_{2r+4\ell+4}^{c_J(b)}(t)\right) + q^{-\lfloor\frac{t-2r}{4}\rfloor} J_{2r+4\lfloor\frac{t-2r}{4}\rfloor}^{c_J(b)}(t)\right).
\end{equation}
Each difference of $J$-Bessel functions in the above sum can be approximated using \eqref{eq. J difference}. Moreover, we note that the second line of \eqref{eq. J difference} is equal to a quotient of two polynomials in the variable $t$, both of degree 4, which implies that their limit as $t\to\infty$ is a nonzero constant. Therefore, for $\theta\in(0,\pi)$ such that $\cos\theta=(1+4b\sqrt{q})^{-1/2}$ we have
$$
    W_{q+1}^{b(q+1-2\sqrt{q}),b}(r;t)\sim\frac{(-1)^rq^{-r/2}}{\sqrt{\pi t\sin\theta}}\left( 1+4b\sqrt{q}
\right) ^{\frac{t}{2}}\cos \left( \left( t+\frac{1}{2}\right)
\theta -\frac{\pi}{4} \right)\cdot C_b(r;t)$$
where
$$
C_b(r;t)= \sum_{\ell=0}^{\lfloor\frac{t-2r}{4}\rfloor -1} q^{-\ell}\frac{(c_J(b)/2)_{2r+4\ell}t^{2r+4\ell}}{(c_J(b)/2)^{2r+4\ell}(t+1)_{2r+4\ell}}\cdot \left(\frac{P_4(b,q,r,\ell;t)}{Q_4(b,q,r,\ell;t)}+ O\left(\frac{1}{t^{3/2}}\right)\right) + O(q^{-t/4}),
$$
as $t\to\infty$, where $P_4(b,q,r,\ell;t)$ and $Q_4(b,q,r,\ell;t)$ are polynomials of degree 4 in $t$ with coefficients depending on $b,\,q,\,r,\,\ell$. These polynomials can be explicitly deduced from the second line of \eqref{eq. J difference} with $c=c_J(b)$ and $n=r+2\ell$, but their explicit expressions are not needed for the proof. 

The term $O(q^{-t/4})$ in the above display stems from the last term in the expression \eqref{eq. W different form} combined with the asymptotic behavior of the $J$-Bessel function described in \cite[Theorem 2]{BSS24}. The statement now follows from the fact that $C_b(r;t)$ is nonzero for sufficiently large $t$.
 \end{proof}

\begin{remark}
    The amplitude of the oscillations for $W_{q+1}^{b(q+1-2\sqrt{q}),b}(r;t)$ grows as $(1+\lambda_{\max})^{t/2}$, where $\lambda_{\max} =\sup\sigma_2(\Delta_{q+1}^{b(q+1-2\sqrt{q}),b}) =4b\sqrt{q}$ is the maximum of the continuous spectrum of the generalized Laplacian $\Delta_{q+1}^{b(q+1-2\sqrt{q}),b}$. In view of the spectral resolution of the wave operator on the tree, this behavior is expected. Namely, the operator $\Delta_{q+1}^{b(q+1-2\sqrt{q}),b}$ has purely continuous spectrum $\sigma_2(\Delta_{q+1}^{b(q+1-2\sqrt{q}),b})=[0, 4b\sqrt{q}]$, and hence the spectral resolution of the corresponding wave operator is
    \begin{align} \nonumber
    \mathcal{W}_t^{b(q+1-2\sqrt{q}),b}&=\int\limits_{[0, 4b\sqrt{q}]}\left(\sum_{k=0}^{\lfloor\frac{t}{2}\rfloor} (-1)^k\binom{t}{2k}\lambda^k\right)\rho_{q+1,b}(d\lambda)=\int\limits_{[0, 4b\sqrt{q}]}\mathrm{Re}\left( 1+i\sqrt{\lambda} \right)^t \rho_{q+1,b}(d\lambda)\\
    &=\int\limits_{[0, 4b\sqrt{q}]}\left( 1+\lambda \right)^{t/2} \cos(t\arctan\sqrt{\lambda})\rho_{q+1,b}(d\lambda).\label{eq. spectal resolution}
    \end{align}
Here $\rho_{q+1,b}$ denotes the spectral measure associated with the operator $\Delta_{q+1}^{b(q+1-2\sqrt{q}),b}$ on the tree $T_{q+1}$. The spectral resolution \eqref{eq. spectal resolution} indicates that the amplitude of the oscillations of the wave kernel $W_{t}^{b(q+1-2\sqrt{q}),b}$ is bounded by $(1+4b\sqrt{q})^{t/2}$.
\end{remark}

\subsection{Large time asymptotic behavior, general case}

In this subsection, we discuss the large-time asymptotic behavior of the wave kernel $W_{q+1}^{a,b}$ for a general choice of real parameters $a,\,b$, with $b\neq 0$, such that \eqref{eq. main ass a,b} holds. Without loss of generality, we assume $b>0$. The spectral resolution of the corresponding wave operator is
\begin{equation}\label{eq. spectral gen}
\mathcal{W}_t^{a,b}=\int\limits_{[\, b(q+1-2\sqrt{q})-a,\; b(q+1+2\sqrt{q})-a \,]}\left( 1+\lambda \right)^{t/2}\cos(t\arctan\sqrt{\lambda}) \rho_{q+1,a,b}(d\lambda),
\end{equation}
where $\rho_{q+1,a,b}$ is the spectral measure associated with the operator $\Delta_{q+1}^{a,b}$. Given that the spectral resolution is very similar to the boundary case, and $b(q+1-2\sqrt{q})-a\geq 0$, a similar asymptotic behavior as $t\to\infty$ is expected.

The large-time asymptotics for $W_{q+1}^{a,b}(r;t)$ can be deduced by using the explicit expression \eqref{wave_kernel_a}, combined with the asymptotic formula
$$
I_n^c(t)\sim \frac{(\mathrm{sgn}(c))^n}{\sqrt{2\pi|c|t}}(1+|c|)^{t+1/2},\quad\text{  as  } t\to\infty
$$
proved in \cite[Proposition 3.5]{CHJSV} for $n,\,t \in\mathbb{N}_0$ and real nonzero $c$, where $\mathrm{sgn}(c)$ denotes the sign of $c$. 

We will not carry out the details here. Instead, Figures~\ref{fig:012} and~\ref{fig:013} illustrate the large-time behavior of the wave kernels $W^{0,\,1}_{3}(r;t)$ and $W^{0,\,1}_{4}(r;t)$ associated with the standard combinatorial Laplacian on the trees $T_3$ and $T_4$, respectively. The figures show very rapid exponential growth in magnitude together with strong sign oscillations, indicating severe instability of the discrete-time wave kernel for the combinatorial Laplacian. Selected numerical values for the $T_3$ case are listed in Table~\ref{tab:WR_q2_standard} in the Appendix.
\begin{figure}[h!]
    \centering
    \includegraphics[width=0.7\linewidth]{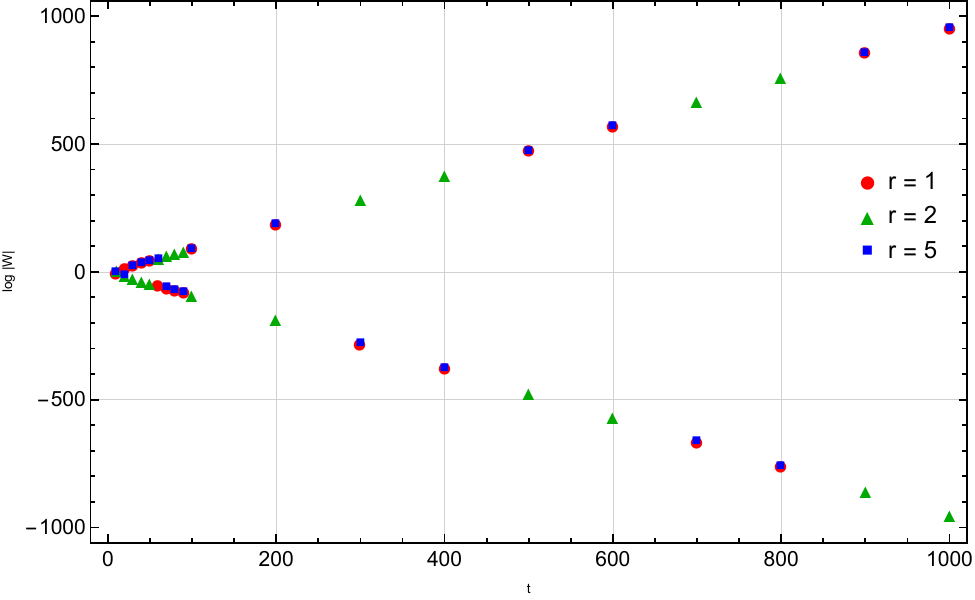}
    \caption{Signed logarithmic transform of $W^{0,\,1}_{3}(r;t)$ for $r=1,2,5$.}
    \label{fig:012}
\end{figure}

\begin{figure}[h!]
    \centering
    \includegraphics[width=0.7\linewidth]{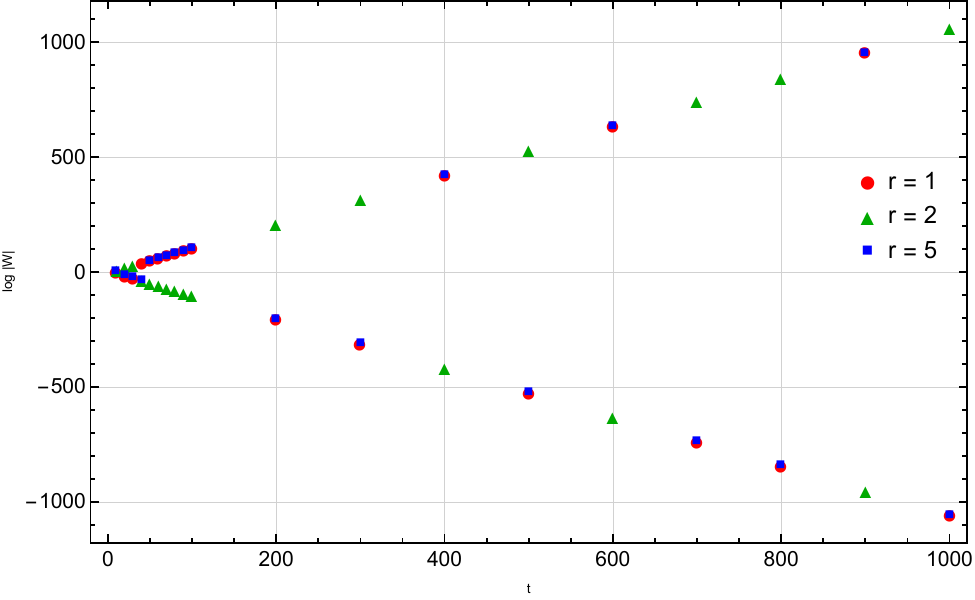}
    \caption{Signed logarithmic transform of $W^{0,\,1}_{4}(r;t)$ for $r=1,2,5$.}
    \label{fig:013}
\end{figure}

 The growth in this case is significantly faster than in the boundary case $a=b(q+1-2\sqrt{q})$. That is a consequence of the fact that the magnitude of oscillations of $W^{a,\,b}_{q+1}(r;t)$ for large $t$ is approximately $(1+b(q+1+2\sqrt{q})-a)^{t/2}$, which, for $a=0$ and $b=1$, equals $(2+q+2\sqrt{q})^{t/2}$, a number that is significantly larger for $q\geq 2$ than the magnitude $(1+4\sqrt{q})^{t/2}$ in the boundary case $a=q+1-2\sqrt{q}$ and $b=1$.

\medskip

The largest eigenvalue for the normalized/probabilistic Laplacian is $\lambda_{\max} = 1+\frac{2\sqrt{q}}{q+1}$, which is significantly smaller than $q+1+2\sqrt{q}$, the largest eigenvalue of the combinatorial Laplacian on $T_{q+1}$. Hence, the oscillations of the wave kernel $W_{q+1}^{0,1/(q+1)}$ associated with the probabilistic Laplacian exhibit smaller, yet still exponential, growth, as illustrated in Figures~\ref{fig:WR0132} and~\ref{fig:WR0143}.

\begin{figure}[h!]
    \centering
    \includegraphics[width=0.7\linewidth]{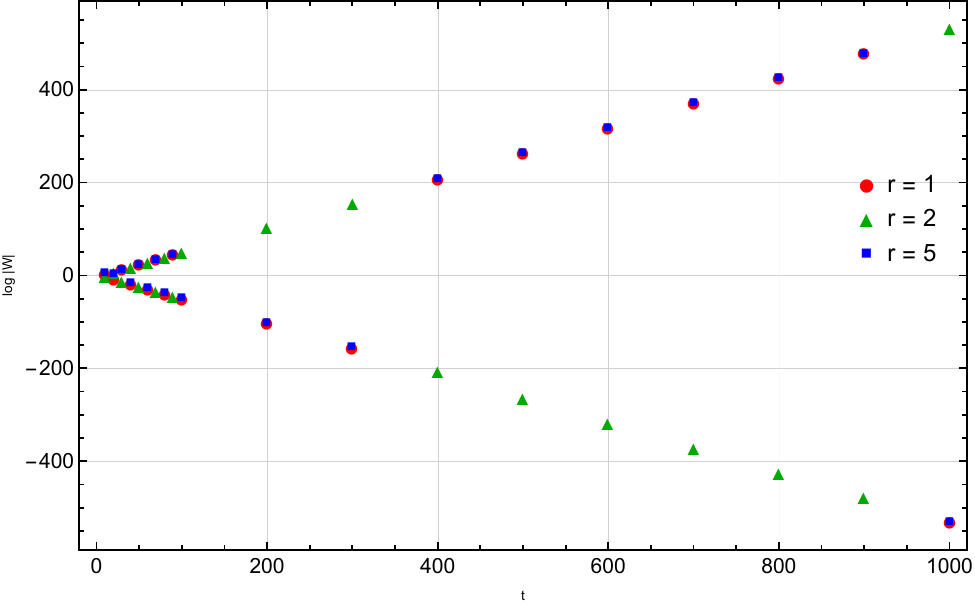}
    \caption{Signed logarithmic transform of $W^{0,\,\frac{1}{3}}_{3}(r;t)$ for $r=1,2,5$.}
    \label{fig:WR0132}
\end{figure}

\begin{figure}[h!]
    \centering
    \includegraphics[width=0.7\linewidth]{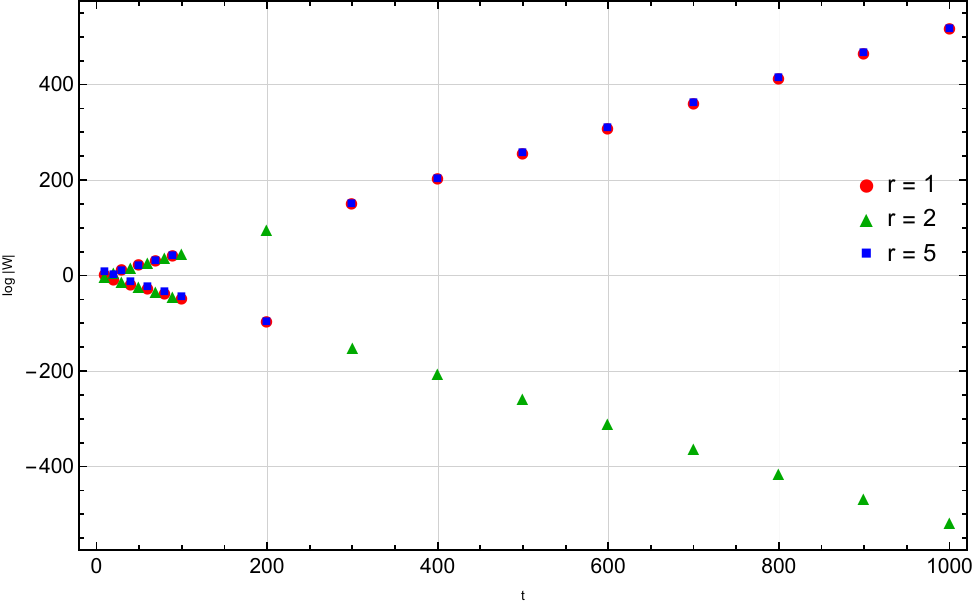}
    \caption{Signed logarithmic transform of $W^{0,\,\frac{1}{4}}_{4}(r;t)$ for $r=1,2,5$.}
    \label{fig:WR0143}
\end{figure}

\medskip

When $a$ and $b$ are chosen so that the largest eigenvalue is close to $0$ (yet positive, due to the bound \eqref{eq. main ass a,b}), 
the magnitude of the oscillations is smaller, as illustrated in Table \ref{tab:WR_smallparams_q4} in the Appendix and shown in Figure \ref{1_100} for the case where $\lambda_{\max}=\frac{2}{25}$. We note that, when $a=b$, condition \eqref{eq. main ass a,b} is fulfilled and $\lambda_{\max}=a(q+2\sqrt{q})$. Therefore, for any given $\delta>0$, by taking $a=b\in (0,\frac{\delta}{q+2\sqrt{q}})$, the maximal $L^2$-eigenvalue of $\Delta_{q+1}^{a,b}$ is less than $\delta$. In that case, the amplitude of oscillations of $W_{q+1}^{a,b}$ grows exponentially, at the rate less than $(1+\delta)^{t/2}$.
\begin{figure}[h!]
    \centering
    \includegraphics[width=0.7\linewidth]{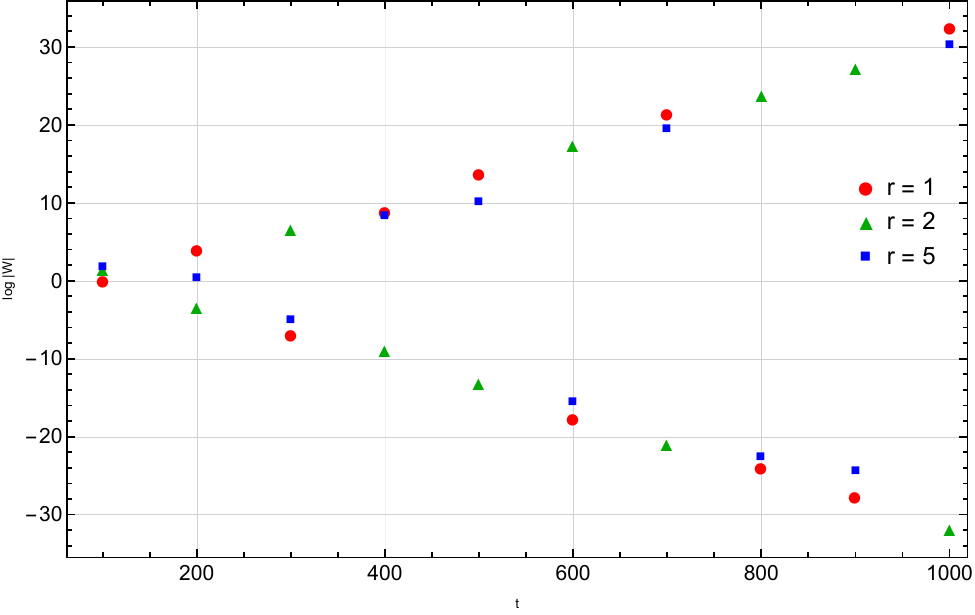}
    \caption{Signed logarithmic transform of $W^{1/100,\, 1/100}_{5}(r;t)$ for $r=1,2,5$.}
    \label{1_100}
\end{figure}

\subsection{Large $q$ asymptotic behavior: probabilistic Laplacian, boundary case}

In physics applications, Bethe lattices with $q=2$ and $q=3$ are the most common, so in the previous examples, we illustrated the behavior of the wave kernels on $3$-tree and $4$-tree. However, it is interesting to observe that, as $q\to\infty$, the wave kernel $W_{q+1}^{1-\frac{2\sqrt{q}}{q+1}, \frac{1}{q+1}}$ associated with the probabilistic Laplacian in the boundary case tends to zero (meaning that, as the edges in the graph become dense, the wave kernel tends to zero for fixed time and fixed radial distance). More precisely, we have the following corollary.

\begin{corollary}
With the notation as above, we have that for fixed $r,\, t \in\mathbb{N}_0$ with $t\geq 2r$
\begin{equation}\label{eq. asymp large q}
W^{1-\frac{2\sqrt{q}}{q+1}, \frac{1}{q+1}}_{q+1}(r;t)=\binom{t}{2r}(q+1)^{-r}\left(1+O\left(\frac{1}{\sqrt{q}}\right)\right),\quad\text{ as } q\to\infty.
\end{equation}
\end{corollary}
\begin{proof}
Recall the expression \eqref{eq:J-form}. When $b=\frac{1}{q+1}$ and $a=1-\frac{2\sqrt{q}}{q+1}$, the constant $c_J$ equals $2\left(\frac{\sqrt{q}}{q+1}\right)^{\frac{1}{2}}$. From the definition \eqref{eq:def-J-Bessel} of the discrete $J$-Bessel function with $c=c_J$, it follows directly that for $t\geq 2r$ we have
$$
J_{2r+4\ell}^{c_J}(t)=\binom{t}{2r+4\ell}\left(\frac{\sqrt{q}}{q+1}\right)^{r+2\ell}\left(1+O\left(\frac{\sqrt{q}}{q+1} \right)\right),\quad\text{as  }q\to\infty,
$$
for all $\ell=0,\ldots,\lfloor\frac{t-2r}{4}\rfloor$.
This, combined with \eqref{eq:J-form} with $b=\frac{1}{q+1}$ yields \eqref{eq. asymp large q}.
\end{proof}

The asymptotic behavior of the  wave kernel $W_{q+1}^{1-\frac{2\sqrt{q}}{q+1}, \frac{1}{q+1}}$ associated with the probabilistic Laplacian as $q\to\infty$ is illustrated in Figure \ref{fig:Wb-q-asymp}, and the corresponding numerical values are given in Table \ref{tab:Wb_q} in the Appendix.
\begin{figure}[h!]
\centering
\includegraphics[width=0.65\textwidth]{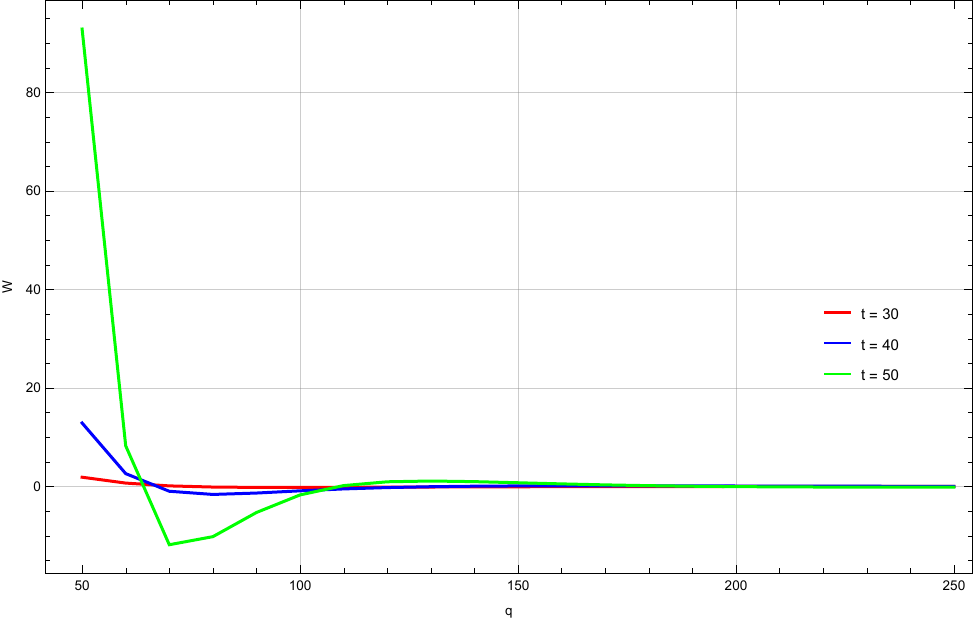}
\caption{Asymptotic behavior of the  wave kernel $W^{1-\frac{2\sqrt{q}}{q+1}, \frac{1}{q+1}}_{q+1}(2;t)$ as a function of  $q$ for $t=30,40,50$.}
\label{fig:Wb-q-asymp}
\end{figure}

\section{Fundamental solutions when  $\Delta_{q+1}^{a,b}$ is not necessarily semipositive}

In this section, we study the situation in which the parameters $a,b$ ($b\neq 0$) are arbitrary, meaning that \eqref{eq. main ass a,b} is not assumed, hence the spectrum of the operator $\Delta_{q+1}^{a,b}$ may be negative. The results of Sections 2 and 3 remain valid for any $a,b$ ($b\neq 0$) as assumption \eqref{eq. main ass a,b} was not posed there. 

In Section 4, the assumption \eqref{eq. main ass a,b} is used in the proof only to ensure that $d_{a,b}\geq 2|b|\sqrt{q}>0$, which yields that $c_{a,b}$ is well-defined. It is evident from the proof of Theorem \ref{Tqkernel_a} in Section 4 that all calculations remain valid if $d_{a,b}<0$. In other words, Theorem \ref{Tqkernel_a} holds true with the assumption \eqref{eq. main ass a,b} replaced by $d_{a,b}\neq 0$.  
When $d_{a,b}=0$, we have $\Delta_{q+1}^{a,b}=-bA_{{q+1}}$, where $A_{{q+1}}$ is the adjacency operator on the tree. We have the following proposition.

\begin{proposition} \label{prop. arb param} For any choice of real numbers $a,\, b$ with $b\neq 0$ and $d_{a,b}=b(q+1)-a \neq 0$,  \eqref{wave_kernel_a} and \eqref{wave_kernel_b} are solutions to \eqref{eq. wave on tree} with initial conditions \eqref{cond1} and \eqref{cond on V}, respectively. 

Moreover, when $a=b(q+1)$, the solution to equation \eqref{eq. wave on tree}, satisfying the initial condition \eqref{cond1}, for $x\in T_{q+1}$ with $|x|=r$, is given by
\begin{equation*}
W^{b(q+1),b}_{q+1}(r;t)
=
b^r\sum_{m=0}^{\lfloor\frac{t-2r}{4}\rfloor} \binom{t}{2r+4m}b^{2m}\sum_{j=0}^m q^j\left(\binom{r+2m}{j}-\binom{r+2m}{j-1}\right).
\end{equation*}
Moreover, the solution to \eqref{eq. wave on tree}, satisfying the initial condition \eqref{cond on V}, for $x\in T_{q+1}$ with $|x|=r$, is
\begin{equation*}
V^{b(q+1),b}_{q+1}(r;t)
=
b^r\sum_{m=0}^{\lfloor\frac{t-2r-1}{4}\rfloor} \binom{t}{2r+4m+1}b^{2m}\sum_{j=0}^m q^j\left(\binom{r+2m}{j}-\binom{r+2m}{j-1}\right),
\end{equation*}
where we use the convention that $\binom{r+2m}{-1}=0$.
\end{proposition}
\begin{proof}
When $d_{a,b} \neq 0$, by differentiating twice with respect to time and using properties of the $I$-Bessel functions from \cite[Lemma 3.1]{CHJSV} it is straightforward to check that \eqref{wave_kernel_a} and \eqref{wave_kernel_b} are solutions to \eqref{eq. wave on tree} with initial conditions \eqref{cond1} and \eqref{cond on V}, respectively. 

Let us assume now that $d_{a,b}=0$. Let $q\geq 2$. We follow the steps of the proof of Theorem \ref{Tqkernel_a}. When $d_{a,b}=0$, using the same notation, we have
    $$
    \widetilde g_{b(q+1),b}= -2b\sqrt{q}\cos(z\log q).
    $$
    Equation \eqref{w_t_1} in this case becomes
$$
W^{b(q+1),b}_{q+1}\left(r;t\right)
=\frac{1}{2\pi}q^{-\frac{r}{2}}\sum_{k=0}^{\lfloor \frac{t}{2} \rfloor} (-1)^k h_{2k}\left(t,0\right)\,(-2b\sqrt{q})^k
\int_{-\pi}^{\pi} \cos^k u\frac{e^{iu|x|}-e^{iu(|x|+2)}}{1-\frac{1}{q}e^{2iu}} \,du.
$$
Following the steps of the proof of Theorem \ref{Tqkernel_a} and applying Lemma \ref{lem:integrali} (a), we immediately deduce that
$$
W^{b(q+1),b}_{q+1}\left(r;t\right)
=q^{-\frac{r}{2}}\sum_{ \substack{k=0 \\ k-r\,\text{even}} }^{\lfloor \frac{t}{2} \rfloor}  h_{2k}\left(t,0\right)\,(b\sqrt{q})^k \sum_{\ell=0}^{ \frac{k-r}{2}}q^{-\ell}\left[\binom{k}{\frac{k-r}{2}-\ell}-\binom{k}{\frac{k-r}{2}-\ell-1}\right].
$$
The statement follows by the change of variables $k=r+2m$ and $j=m-\ell$. The proof of the expression for $V^{b(q+1),b}_{q+1}\left(r;t\right)$ is completely analogous; it follows by replacing $h_{2k}(t,0)$ with $h_{2k+1}(t,0)$.

The case $q=1$ can be treated analogously by using the Fourier transform on $\mathbb Z$.
\end{proof}

\begin{remark} \label{rem. why semipos} The assumption \eqref{eq. main ass a,b} is imposed in order to ensure the semipositivity of $\Delta_{q+1}^{a,b}$. In this range of parameters $a,\, b$, the fundamental solutions $W_{q+1}^{a,b}$ and $V_{q+1}^{a,b}$ oscillate and hence behave like \emph{waves}. If the spectrum of $\Delta_{q+1}^{a,b}$ contains negative real numbers, the fundamental solutions $W_{q+1}^{a,b}$ and $V_{q+1}^{a,b}$ to the equation \eqref{wave_eq_disc} satisfying the initial conditions \eqref{cond1} and \eqref{cond on V}, respectively are given by \eqref{wave_kernel_a} and \eqref{wave_kernel_b} may exhibit fundamentally different asymptotic behavior as $t\to\infty$. Namely, if, for example $b>0$ and $b(q+1+2\sqrt{q})-a< 0$, the spectral resolution \eqref{eq. spectral gen} of the wave operator becomes
\begin{equation*}
\mathcal{W}_t^{a,b}= \sum_{k=0}^{\lfloor\frac{t}{2}\rfloor}\binom{t}{2k}\int\limits_{[\, b(q+1-2\sqrt{q})-a,\; b(q+1+2\sqrt{q})-a \,]}|\lambda|^k\rho_{q+1,a,b}(d\lambda).
\end{equation*}
Therefore, instead of oscillations with exponentially growing amplitude, as $t\to\infty$ one observes exponential growth without oscillations. In other words, the solution $W_{q+1}^{a,b}(r;t)$ no longer describes wave propagation; rather, it describes exponential growth, which can occur in situations where negative-frequency modes arise in wave propagation (due, for example, to radiation effects or the Hawking effect; see \cite{HB-L16}). Exponential growth of the function $W_{3}^{2,1/3}(r;t)$ for $r\in\{1,2,5\}$ is visualized in Figure \ref{fig:02132}.

\begin{figure}[h!]
    \centering
    \includegraphics[width=0.8\linewidth]{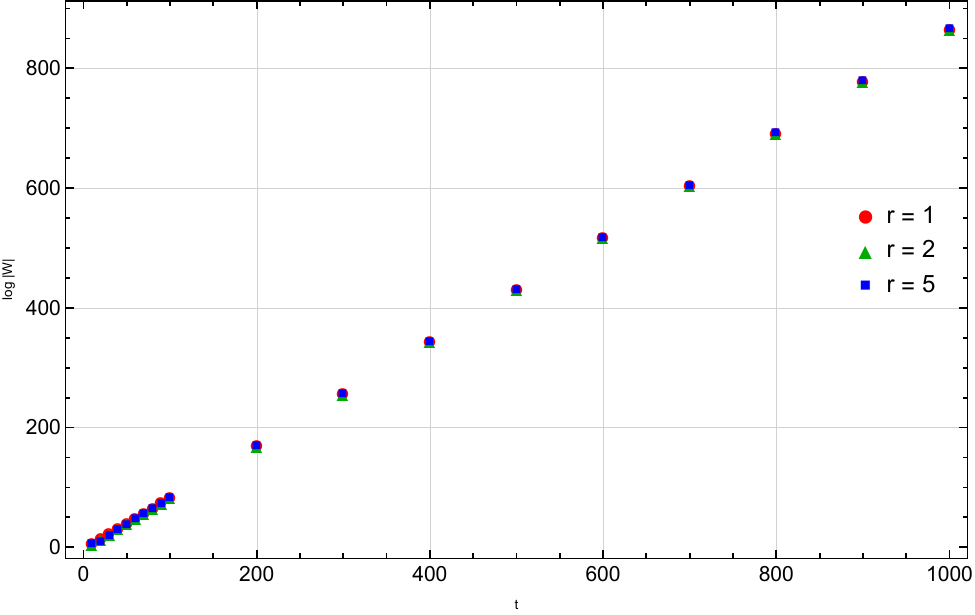}
    \caption{Signed logarithmic transform of $W_{3}^{2,1/3}(r;t)$ for $r=1,2,5$.}
    \label{fig:02132}
\end{figure}

In situations when $b>0$, $b(q+1-2\sqrt{q})-a <0$ and $b(q+1+2\sqrt{q})-a >0$ the spectral resolution is
\begin{equation*}
\mathcal{W}_t^{a,b}=\sum_{k=0}^{\lfloor\frac{t}{2}\rfloor}\binom{t}{2k}\int\limits_{[\, b(q+1-2\sqrt{q})-a,\; 0 \,]}|\lambda|^k\rho_{q+1,a,b}(d\lambda)+\int\limits_{[\, 0,\; b(q+1+2\sqrt{q})-a \,]}\mathrm{Re}\left( 1+i\sqrt{\lambda} \right)^t \rho_{q+1,a,b}(d\lambda)
\end{equation*}
and there is a combination of exponential growth and oscillations with one of the terms dominating, depending on the parameters. Further study of the fundamental solutions \eqref{wave_kernel_a} and \eqref{wave_kernel_b} in the case when the operator $\Delta_{q+1}^{a,b}$ is no longer semipositive is left for future studies.
\end{remark}

\subsection*{Data Availability}
All numerical data and figures were generated using Wolfram Mathematica. Additional numerical data are available from the authors upon reasonable request.

\clearpage

\appendix

\section{Selected numerical tables}\label{app:selected-tables}

\setcounter{table}{0}
\renewcommand{\thetable}{A.\arabic{table}}

This appendix contains selected numerical tables illustrating the main asymptotic regimes discussed in Section~\ref{sec. asymptotics}.

\begin{table}[h!]
\centering
\caption{Numerical values of $W_3^{3-2\sqrt{2},\,1}(r;t)$ for $r\in\{1,2,5,10\}$.}
\label{tab:WRBoundary_q2}
\setlength{\tabcolsep}{7pt}
\renewcommand{\arraystretch}{1.15}
\scalebox{0.8}{
\begin{tabular}{rrrrr}
\hline
$t$ & $r=1$ & $r=2$ & $r=5$ & $r=10$ \\
\hline
5   & $-42.4264$ & $10$ & $0$ & $0$ \\
10  & $-7324.21$ & $3440$ & $5.65685$ & $0$ \\
15  & $-1.2817\times 10^{6}$ & $752030$ & $-14068.6$ & $0$ \\
20  & $-2.26479\times 10^{8}$ & $1.49733\times 10^{8}$ & $-9.11643\times 10^{6}$ & $32$ \\
25  & $-3.96591\times 10^{10}$ & $2.84365\times 10^{10}$ & $-3.21547\times 10^{9}$ & $-656640$ \\
30  & $-6.73438\times 10^{12}$ & $5.16247\times 10^{12}$ & $-8.89089\times 10^{11}$ & $-1.15958\times 10^{8}$ \\
35  & $-1.07281\times 10^{15}$ & $8.80869\times 10^{14}$ & $-2.13891\times 10^{14}$ & $2.77897\times 10^{11}$ \\
40  & $-1.49319\times 10^{17}$ & $1.35106\times 10^{17}$ & $-4.63936\times 10^{16}$ & $2.38769\times 10^{14}$ \\
45  & $-1.40705\times 10^{19}$ & $1.62736\times 10^{19}$ & $-9.09769\times 10^{18}$ & $1.16848\times 10^{17}$ \\
50  & $1.02399\times 10^{21}$ & $5.0767\times 10^{20}$ & $-1.56275\times 10^{21}$ & $4.36418\times 10^{19}$ \\
55  & $1.13472\times 10^{24}$ & $-5.8751\times 10^{23}$ & $-2.08152\times 10^{23}$ & $1.37008\times 10^{22}$ \\
60  & $4.50761\times 10^{26}$ & $-2.87559\times 10^{26}$ & $-8.13917\times 10^{24}$ & $3.76857\times 10^{24}$ \\
65  & $1.4356\times 10^{29}$ & $-9.88731\times 10^{28}$ & $8.03488\times 10^{27}$ & $9.21035\times 10^{26}$ \\
70  & $4.1333\times 10^{31}$ & $-2.96981\times 10^{31}$ & $4.22777\times 10^{30}$ & $1.97887\times 10^{29}$ \\
75  & $1.12182\times 10^{34}$ & $-8.28945\times 10^{33}$ & $1.52733\times 10^{33}$ & $3.531\times 10^{31}$ \\
80  & $2.92629\times 10^{36}$ & $-2.20717\times 10^{36}$ & $4.77372\times 10^{35}$ & $4.07394\times 10^{33}$ \\
85  & $7.41286\times 10^{38}$ & $-5.68161\times 10^{38}$ & $1.37738\times 10^{38}$ & $-3.80558\times 10^{35}$ \\
90  & $1.83468\times 10^{41}$ & $-1.42475\times 10^{41}$ & $3.77135\times 10^{40}$ & $-4.73269\times 10^{38}$ \\
95  & $4.45293\times 10^{43}$ & $-3.49652\times 10^{43}$ & $9.93953\times 10^{42}$ & $-2.14123\times 10^{41}$ \\
100 & $1.06223\times 10^{46}$ & $-8.42146\times 10^{45}$ & $2.5422\times 10^{45}$ & $-7.64996\times 10^{43}$ \\
\hline
\end{tabular}}
\end{table}

\begin{table}[h!]
\centering
\caption{Numerical values of $W^{0,\,1}_{3}(r;t)$ for
$r\in\{1,2,5,10\}$.}
\label{tab:WR_q2_standard}
\setlength{\tabcolsep}{7pt}
\renewcommand{\arraystretch}{1.15}
\scalebox{0.8}{
\begin{tabular}{rrrrr}
\hline
$t$ & $r=1$ & $r=2$ & $r=5$ & $r=10$ \\
\hline
10   & $-1.171\times 10^{3}$ & $6.90\times 10^{2}$ & $1$ & $0$ \\
20   & $5.684\times 10^{6}$ & $-1.5224\times 10^{6}$ & $-4.344\times 10^{5}$ & $1$ \\
30   & $2.0195\times 10^{11}$ & $-1.25593\times 10^{11}$ & $4.98882\times 10^{9}$ & $9.63667\times 10^{6}$ \\
40   & $2.48235\times 10^{15}$ & $-1.79246\times 10^{15}$ & $2.72466\times 10^{14}$ & $4.55025\times 10^{11}$ \\
50   & $1.30607\times 10^{19}$ & $-1.09917\times 10^{19}$ & $3.51674\times 10^{18}$ & $-3.02768\times 10^{16}$ \\
60   & $-1.49617\times 10^{23}$ & $9.59936\times 10^{22}$ & $1.45796\times 10^{21}$ & $-9.87938\times 10^{20}$ \\
70   & $-4.74011\times 10^{27}$ & $3.55348\times 10^{27}$ & $-7.35673\times 10^{26}$ & $-3.69661\times 10^{24}$ \\
80   & $-6.11176\times 10^{31}$ & $4.84637\times 10^{31}$ & $-1.43108\times 10^{31}$ & $3.25444\times 10^{29}$ \\
90   & $-3.16391\times 10^{35}$ & $2.75546\times 10^{35}$ & $-1.24266\times 10^{35}$ & $8.08368\times 10^{33}$ \\
100  & $4.76749\times 10^{39}$ & $-3.47521\times 10^{39}$ & $4.88503\times 10^{38}$ & $7.05503\times 10^{37}$ \\
200  & $2.4411\times 10^{81}$ & $-2.05909\times 10^{81}$ & $8.93997\times 10^{80}$ & $-1.06456\times 10^{80}$ \\
300  & $-4.13444\times 10^{122}$ & $3.41611\times 10^{122}$ & $-1.30841\times 10^{122}$ & $7.70748\times 10^{120}$ \\
400  & $-2.6984\times 10^{164}$ & $2.32815\times 10^{164}$ & $-1.15228\times 10^{164}$ & $2.17352\times 10^{163}$ \\
500  & $6.01955\times 10^{205}$ & $-5.10468\times 10^{205}$ & $2.2891\times 10^{205}$ & $-3.00535\times 10^{204}$ \\
600  & $4.19927\times 10^{247}$ & $-3.65175\times 10^{247}$ & $1.88938\times 10^{247}$ & $-4.12308\times 10^{246}$ \\
700  & $-1.05611\times 10^{289}$ & $9.05899\times 10^{288}$ & $-4.34065\times 10^{288}$ & $7.33059\times 10^{287}$ \\
800  & $-7.59763\times 10^{330}$ & $6.63282\times 10^{330}$ & $-3.50923\times 10^{330}$ & $8.22346\times 10^{329}$ \\
900  & $2.04528\times 10^{372}$ & $-1.76575\times 10^{372}$ & $8.77682\times 10^{371}$ & $-1.68400\times 10^{371}$ \\
1000 & $1.49712\times 10^{414}$ & $-1.31034\times 10^{414}$ & $7.02629\times 10^{413}$ & $-1.71737\times 10^{413}$ \\
\hline
\end{tabular}}
\end{table}

\newpage

\begingroup
\setlength{\textfloatsep}{7pt}
\setlength{\floatsep}{7pt}
\setlength{\intextsep}{7pt}
\setlength{\abovecaptionskip}{4pt}
\setlength{\belowcaptionskip}{4pt}
\renewcommand{\arraystretch}{1.15}

\begin{table}[h!]
\centering
\caption{Numerical values of $W^{1/100,\, 1/100}_{5}(r;t)$ for $r\in\{1,2,5,10\}$.}
\label{tab:WR_smallparams_q4}
\setlength{\tabcolsep}{7pt}
\scalebox{0.8}{\begin{tabular}{rrrrr}
\hline
$t$ & $r=1$ & $r=2$ & $r=5$ & $r=10$ \\
\hline
10  & $0.29379$ & $0.018529$ & $1.0\times 10^{-10}$ & $0$ \\
20  & $-0.220978$ & $0.140936$ & $1.55901\times 10^{-5}$ & $1.0\times 10^{-20}$ \\
30  & $-0.212359$ & $-0.167347$ & $0.00139382$ & $2.75345\times 10^{-13}$ \\
40  & $0.442945$ & $-0.0704245$ & $0.0116485$ & $9.44178\times 10^{-10}$ \\
50  & $-0.683952$ & $0.272967$ & $0.00530552$ & $1.92192\times 10^{-7}$ \\
60  & $0.811726$ & $-0.47243$ & $-0.0426036$ & $7.783\times 10^{-6}$ \\
70  & $-0.784192$ & $0.586904$ & $0.0349734$ & $9.6707\times 10^{-5}$ \\
80  & $0.508425$ & $-0.558087$ & $0.0133909$ & $3.92562\times 10^{-4}$ \\
90  & $0.0938033$ & $0.308333$ & $-0.0913328$ & $1.05859\times 10^{-4}$ \\
100 & $-1.07637$ & $0.226163$ & $0.168151$ & $-0.00153$ \\
200 & $48.1217$ & $-29.5156$ & $1.44998$ & $0.0469967$ \\
300 & $-1083.83$ & $770.338$ & $-158.91$ & $-2.36607$ \\
400 & $6760.78$ & $-7086.78$ & $3710.57$ & $-139.181$ \\
500 & $827883$ & $-490317$ & $22340.2$ & $6457.87$ \\
600 & $-5.23032\times 10^{7}$ & $3.45455\times 10^{7}$ & $-5.87278\times 10^{6}$ & $-29576$ \\
700 & $1.79073\times 10^{9}$ & $-1.24164\times 10^{9}$ & $2.79428\times 10^{8}$ & $-8.52811\times 10^{6}$ \\
800 & $-2.73523\times 10^{10}$ & $2.08889\times 10^{10}$ & $-6.81709\times 10^{9}$ & $4.69778\times 10^{8}$ \\
900 & $-1.11529\times 10^{12}$ & $6.70107\times 10^{11}$ & $-4.17207\times 10^{10}$ & $-1.10533\times 10^{10}$ \\
1000 & $1.10755\times 10^{14}$ & $-7.35744\times 10^{13}$ & $1.34534\times 10^{13}$ & $-1.70982\times 10^{11}$ \\
\hline
\end{tabular}}
\end{table}

\begin{table}[ht]
\centering
\caption{Numerical values of 
$W^{1-\frac{2\sqrt{q}}{q+1}, \frac{1}{q+1}}_{q+1}(2;t)$ for $t \in \{30,40,50\}$.}
\label{tab:Wb_q}
\scalebox{0.7}{\begin{tabular}{rrrr}
\hline
$q$ & $t=30$ & $t=40$ & $t=50$ \\
\hline
50  & 1.88376  & 12.8866  & 92.8707 \\
60  & 0.691767 & 2.61305  & 8.2564 \\
70  & 0.120222 & -0.967351 & -11.8241 \\
80  & -0.112896 & -1.61512 & -10.1832 \\
90  & -0.186473 & -1.33229 & -5.27805 \\
100 & -0.190592 & -0.880076 & -1.71217 \\
110 & -0.167687 & -0.494147 & 0.184031 \\
120 & -0.136977 & -0.220022 & 0.95437 \\
130 & -0.106705 & -0.0442682 & 1.11092 \\
140 & -0.0800215 & 0.0588285 & 0.985005 \\
150 & -0.0577703 & 0.112617 & 0.760453 \\
160 & -0.0398066 & 0.134785 & 0.528146 \\
170 & -0.0256128 & 0.137696 & 0.326832 \\
180 & -0.0145782 & 0.129597 & 0.168345 \\
190 & -0.00611999 & 0.115808 & 0.0518489 \\
200 & 0.000272425 & 0.0996508 & -0.0286126 \\
210 & 0.00502669 & 0.0831283 & -0.0803555 \\
220 & 0.00849221 & 0.0673857 & -0.110335 \\
230 & 0.0109502 & 0.0530233 & -0.124484 \\
240 & 0.0126248 & 0.0403029 & -0.12756 \\
250 & 0.0136936 & 0.0292816 & -0.123229 \\
\hline
\end{tabular}}
\end{table}
\endgroup
\end{document}